\documentclass[a4paper,unicode,manuscript]{amsart}
\usepackage[utf8]{inputenc}

\usepackage{amsmath,amsfonts,amssymb,amsthm,amscd,latexsym,mathrsfs, mathabx, amsbsy}
\usepackage[french, english]{babel}
\usepackage{mathtools}
\usepackage[usenames, dvipsnames]{color}

\usepackage{enumitem}
\usepackage{fancyhdr}
\usepackage{stmaryrd,xypic, multicol, tikz-cd}

\usepackage{hyperref}
\hypersetup{
    colorlinks,
    linkcolor=Blue,
    filecolor=magenta,      
    urlcolor=cyan,
}

\usepackage{verbatim}
\usepackage{tikz-cd}
\theoremstyle{definition}

\newtheorem{definition}{Definition}[section]
\newtheorem{proposition}[definition]{Proposition}
\newtheorem{lemma}[definition]{Lemma}
\newtheorem{theorem}[definition]{Theorem}
\newtheorem{corollary}[definition]{Corollary}

\newtheorem{remark}[definition]{Remark}

\newtheorem{example}[definition]{Example}

\newtheorem{construction}[definition]{Construction}

\newcommand{\bbD}{\mathbb{D}}

\newcommand{\bbN}{\mathbb{N}}
\newcommand{\bbZ}{\mathbb{Z}}
\newcommand{\bbR}{\mathbb{R}}

\newcommand{\KA}{\mathcal{A}}
\newcommand{\KB}{\mathcal{B}}
\newcommand{\KR}{\mathcal{R}}
\newcommand{\KT}{\mathcal{T}}

\newcommand{\KP}{\mathcal{P}}
\newcommand{\KM}{\mathcal{M}}
\newcommand{\KC}{\mathcal{C}}
\newcommand{\KE}{\mathcal{E}}
\newcommand{\KF}{\mathcal{F}}

\newcommand{\KV}{\mathcal{V}}

\newcommand{\SL}{\mathscr{L}}
\newcommand{\SX}{\mathscr{X}}
\newcommand{\SH}{\mathscr{H}}

\newcommand{\fX}{\mathfrak{X}}
\newcommand{\fU}{\mathfrak{U}}

\newcommand{\fL}{\mathfrak{L}}

\newcommand{\sbullet}{{\scriptscriptstyle\bullet}}

\DeclareMathOperator{\spec}{Spec}
\DeclareMathOperator{\proj}{Proj}

\DeclareMathOperator{\dist}{dist}

\DeclareMathOperator{\an}{\mathrm{an}}
\DeclareMathOperator{\FS}{\mathrm{FS}}

\DeclareMathOperator{\Sp}{\mathrm{sp}}

\DeclareMathOperator{\ddc}{\mathrm{dd}^{c}}
\DeclareMathOperator{\MA}{\mathrm{MA}}

\DeclareMathOperator{\V}{\mathbf{V}}
\DeclareMathOperator{\Cone}{\mathbf{C}}

\DeclareMathOperator{\Div}{\mathrm{div}}

\DeclareMathOperator{\vol}{\mathrm{vol}}

\DeclareMathOperator{\red}{\mathrm{red}}

\DeclareMathOperator{\Sh}{\Check{S}}

\DeclareMathOperator{\Shbd}{\partial_{\mathrm{Sh}}}

\DeclareMathOperator{\bir}{\mathrm{bir}}

\DeclareMathOperator{\supp}{\mathrm{supp}}
\DeclareMathOperator{\Prob}{\mathrm{Prob}}
\DeclareMathOperator{\fin}{\mathrm{fin}}
\DeclareMathOperator{\eq}{\mathrm{eq}}

\DeclareMathOperator{\Min}{\mathrm{Min}}
\DeclareMathOperator{\Frac}{\mathrm{Frac}}

\newcommand{\ndot}{\mathord{\cdot}}
\DeclarePairedDelimiter{\norm}{\lVert}{\rVert}
\DeclarePairedDelimiter{\abs}{\lvert}{\rvert}
\DeclarePairedDelimiter{\algnorm}{\vvvert}{\vvvert}

\author{Yanbo FANG}

\begin{document}
\title{Metrised ample line bundles in non-Archimedean geometry II}

\maketitle


\section{Introduction}

\subsection{Arithmetic Hilbert-Samuel function over a local place}
We continue our study of semipositive metrics $\phi$ on an ample line bundle $L$ of an irreducible projective variety $X$ of dimension $d$ over a non-archimedean (n-A) field $(k,\abs{\ndot})$ in \cite{Fan1}, from the point of view of its Banach algebra of normed sections $\KR(L,\phi)$. Recall that we denote by 
\begin{itemize}
    \item $R_n(L)$ the vector space $H^0(X, L^n)$; $R_{\sbullet}(L)$ the algebra $\oplus_{n\in\bbN} H^0(X, L^n)$;
    \item $\chi(n)$ the geometric Hilbert-Samuel (HS) function $n\mapsto \dim R_n(L)$;
    \item $\norm{\ndot}_{n\phi}$ the sup norm induced by $\phi$ on $R_n(L)$ for every $n\in\bbZ$, and $\algnorm{\ndot}_{\phi}$ the algebra norm on $R_{\sbullet}(L)$ as the orthogonal sum $\obot_{n\in\mathbb{N}}\norm{\ndot}_{n\phi}$;
    \item $\widehat{\chi}(n)$ the arithmetic HS function (over the local place $(k,\abs{\ndot})$), defined as the map $n\mapsto \det(R_n(L),\norm{\ndot}_{n\phi})$ valued in normed lines over the base valued field.
\end{itemize}


The geometric HS formula computes the leading term of $\chi(n)$ (normalised by $\frac{n^d}{d!}$) in the asymptotic expansion for $n\gg 1$ to be the intersection number $c_1(L)^n$. The classical strategy is to use an induction by dimension argument, comparing the difference of $\chi(n)$ for terms in the following exact sequence induced by multiplication by a global section $s\in R_1(L)$ and restriction to its divisor $V(s)$: 
\[0\rightarrow R_n(L)\xrightarrow{\cdot s}R_{n+1}(L)\xrightarrow{\ndot|_{V(s)}}R_{n+1}(L|_{V(s)})\rightarrow 0.\]

We aim to study the asymptotic behavior of the function $\widehat{\chi}(n)$; mimicking the classical strategy, one would like to evaluate its change in the sequence of normed vector spaces
\[(R_n(L),\norm{\ndot}_{n\phi})\xrightarrow{\cdot s}(R_{n+1}(L),\norm{\ndot}_{(n+1)\phi})\xrightarrow{\cdot|_{V(s)}}(R_{n+1}(L|_{V(s)}),\norm{\ndot}_{(n+1)\phi|_{V(s)}}).\]
This sequence is no longer `exact' with respect to norms on them, since the linear maps are not necessarily isometries. Metrisation leads to secondary/arithmetic invariants and our arithmetic HS function is of this nature. Under the identification of the underlying lines $\det R_{n+1}(L)\simeq\det R_{n}(L)\otimes \det R_{n+1}(L|_{V(s)})$, the change of $\widehat{\chi}(n)$ in this sequence is the distorsion of determinant norms
\[\tau_n(L,\phi; s):=\frac{\det\norm{\ndot}_{(n+1)\phi}}{\det\norm{\ndot}_{n\phi}\det\norm{\ndot}_{(n+1)\phi|_{\Div(s)}}}\in\bbR.\]
To estimate it, decompose the distorsion into an equidistribution part and an extension part; both parts can be expressed as operator norms (denoted by $\norm{\ndot}_{\hom}$) naturally induced by $\phi$ 
\[\tau_n(L,\phi;s)=\norm{\bigwedge^{\chi(n)}(\cdot s)}_{\hom,\phi}\cdot \norm{\bigwedge^{\chi_{V(s)}(n+1)}(\ndot|_{V(s)})}_{\hom,\phi}^{-1}.\]

We've estimated in \cite{Fan1} that the second term has an negligible asymptotic contribution in the limit. This article studies the limit of the first term. 

\subsection{Asymptotic formula and equidistribution measure}
Results in special cases or in similar forms are obtained previously (see next part) by other techniques, here we take a direct \emph{non-archimedean analytic} approach to evaluate the normalised limit of $\frac{1}{\chi(n)}\tau_n(L,\phi)$. 

We begin with a class of metrics which we call \emph{Shilov finite}: the Shilov boundary $\Shbd\KR(L,\phi)$ (of its Berkovich spectrum) consists of finitely many birational points. This class includes all semipositive model metrics, but also potentially more, since Shilov finite metric could possibly be singular.

First we can calculate the desired limit from the Banach algebra $\KR(L,\phi)$ using norm reduction techniques in the case of Shilov finite metrics. Expressed on the `disc domain' $\KM(\KR(L,\phi))$, the limit also takes the form of the pairing of the function $\log\abs{s}$ with a probability measure $\mu_{\eq}(\phi)$ which we call the \emph{equidistribution measure}. The measure takes the form
\[\mu_{\eq}(\phi):=\sum_{a\in \Shbd\KR(L,\phi)}\lambda_a\delta_a,\ \lambda_a:=\lim_{n\to\infty}\dim_{\widetilde{k}}[\pi_a\KR(L,\phi)]^{\sim}_n/\dim_{\widetilde{k}}[\KR(L,\phi)]^{\sim}_n\]
where $\pi_a\KR(L,\phi)$ is an analytic localization of $\KR(L,\phi)$ at $a$, and the $\widetilde{k}$-vector spaces $[-]^{\sim}$ are reductions of corresponding normed $k$-vector spaces. The equidistribution measure (pushed-forward to $X^{\an}$ by the natural projection map) is expected to coincide with the Monge-Ampère measure $\MA(\phi)$ if the metric is continuous. 
\begin{theorem}[\ref{T: equidistribution limit}]
    Let $\phi$ be a metric which is Shilov finite. Then it holds
    \[\lim_{n\to\infty}\frac{1}{\chi(n)}\log~ \norm{\bigwedge^{\chi(n)}(\cdot s)}_{\hom,\phi}=\int_{\Cone(L)^{\an}}\log\abs{s}\mu_{\eq}(\phi).\]
\end{theorem}

Then by an approximation argument one extends the result to the case of semipositive continuous metrics:
\begin{corollary}[\ref{T: equidistribution semipositive}]
    Let $\phi$ be a continuous semipositive metric. Then it holds
    \[\lim_{n\to\infty}\frac{1}{\chi(n)}\log~ \norm{\bigwedge^{\chi(n)}(\cdot s)}_{\hom,\phi}=\int_{X^{\an}}\log\abs{s}_{\phi}\MA(\phi).\]
\end{corollary}
We view these results as an arithmetic Hilbert-Samuel formula over the local place $(k,\abs{\ndot})$.

A natural class of Shilov finite metrics are those whose disc domain is a finite union of affinoid domains; if the metrics are assumed to be continuous, by a result of \cite{GK} they turn out to coincide with the class of piecewise linear (semipositive) metrics (Proposition \ref{P: piecewise-linear is equivalent to Stein-Liu}). For general Shilov finite metrics, we show some continuity properties of their equidistribution measures (Proposition \ref{P: surjectivity of equidistribution measure coefficients}) once their supports (projected to $X^{\an}$) is contained in a given finite set; this allows one to solve the equation for Shilov finite metrics with prescribed finite equidistribution measure (Proposition \ref{P: surjectivity of equidistribution measure coefficients}). 

Our results suggest that the study of non-Archimedean Monge-Ampère measure and related equation for semipositive metrics, as developpped in works such as \cite{CLD}\cite{GK} and \cite{BFJ2} \cite{BGJKM}\cite{BGM}, could be decoupled into two steps, each with a more analytic flavor:
\begin{itemize}
    \item first, focus on the equidistribution measure using purely Banach-algebra techniques, which employs its intrinsic integral model; certain properties such as differentiability of the relative volume and orthogonality property for the support of MA measure have counterparts and can be understood more intuitively in this way;
    \item then, to compare this measure to the genuine n-A MA measure, one needs the continuity property of the envelope metric.
\end{itemize}
This article only takes part of the first step. 

\subsection{Relations to other works}
\subsubsection{Complex analytic prototype and global arithmetic HS formula}
In the complex analytic setting, such an equidistribution identity is proved in \cite{AB} in their study of arithmetic HS function over number fields, where the determinant norm distorsion serves as the contribution to the arithmetic intersection number $\widehat{c}_1(\SL,\phi)^{d+1}$ from Archimedean places; for smooth strictly positive metrics, the equidistribution term is decoupled using peak sections adapted to a suitable dissection of the manifold $X^{\an}$ into a number of $\chi(n)$ small pieces; the construction of such sections relies on (analytic) localization techniques concerning spectral theory of $\overline{\partial}$-Laplacian and Bergman kernel estimates. 

The goal of \cite{AB} was to prove a simplified version to the arithmetic HS formula over a number field, originally established by Gillet and Soulé, as a consequence of their powerful arithmetic Riemann-Roch formula \cite{GS}. Let $K$ be a number field and let $(X,L)$ defined over $K$. For ample model line bundle of $L$ over $\mathfrak{o}_K$, equipped with a Hermitian metric $\phi_{\infty}$ of positive curvature on complex fibers, Gillet and Soulé's formula reads (stated in an adelic way, in a simplified version \footnote{In the sense that only the dominant term is kept; the original formula contains more terms in the asymptotic expansion with respect to $n$})
\[\frac{d!}{n^d}\log(\mathrm{vol}(\mathbb{B}(L^{\otimes n},\{n\phi_v\}))/\mathrm{covol}(H^0(L^{\otimes n})))\xrightarrow{n\to\infty}\widehat{c}_1(L,\{\phi_v\})^{d+1} \]
where one considers the adelic space $\mathbb{A}_k\otimes_k H^0(X,L^{\otimes n})$, together with non-Archimedean norms and metrics $\phi_v$ induced by the ample model, the archimedean metric $\phi_{\infty}$, and $\mathbb{B}$ the unit ball in the adelic section space with respect to sup norms from $\{\phi_v\}$ for all places $v$.

This version was extended to semipositive adelic metrics on ample line bundles \cite{Zha}: a semipositive adelic metric on $L$ consists of metrics $\phi_v$ on $L_v$ for each place $v$ of $K$ as described above, but with exceptions on finitely many non-archimedean places $v$ where $\phi_v$ could be a uniform limit of ample model metrics. Our result provides the non-Archimedean local piece for this version of arithmetic HS formula. 

\subsubsection{Non-archimedean equidistribution for model metrics} In the non-archimedean setting, we start with Shilov finite metrics (which include all integral/formal model metrics), and a weak (analytic) localization is realized using norm reduction and HS formula on the reduction section algebra over the residual field; these techniques are specific to the ultrametric world. Note that reduction only yields peak sections in a weak sense: they from a n-A orthogonal basis, with their pics localized to the support of the corresponding equidistribution measure, however each point in the support comes with the residual H-S multiplicity.   



For a metric $\phi$ coming from an ample integral model ($\simeq$ n-A Fubini-Study metric), the above formula has been obtained in \cite{Mor} using relative asymptotic Riemann-Roch formula on the integral model over integral ring. One can reformulate this proof in terms of affinoid algebras similar to certain calculation from \cite{Gub}, see the Appendix B. Our result grows out of an attempt to give an analytic generalisation of that, mimicking the argument by non-archimedean norm reduction in \cite[\S 6.6, 6.9]{CLD}, to include more metrics.

\subsubsection{Non-Archimedean differentiability of relative volume} Note that in litteratures such as \cite{CL}\cite{KT}\cite{CMc}\cite{BE}\cite{BGM}, a related but different distorsion, the relative volume $\vol(L,\phi,\phi+tf)$ of two sup norms corresponding to a continuous semipositive metric $\phi$ and its (small) perturbation by a continuous function $f$ has been thoroughly studied, the key outcome is a differentiability result 
\[ \frac{d}{dt}|_{t=0}\lim\frac{1}{\chi(n)}\log\frac{\det\norm{\ndot}_{n(\phi+tf)}}{\det\norm{\ndot}_{n\phi}}=\int_{X^{\an}}\log\abs{f}\MA(\phi)\]
where $\MA(\phi)$ is the non-Archimedean Monge-Ampère measure constructed in \cite{CL} and in \cite{CLD}\cite{GK}. This limit of relative volume does not directly give the limit of distorsion even if we are allowed to take $f=\log\abs{s}_{\phi}$ using approximations, because the limit procedures are different: our limit for $\tau_n(L,\phi)$ is a diagonal version of the two parameter limit for $\vol_n(L,\phi,\phi+tf)$. To get ours from theirs, some uniformity is needed. The passage has been recently established in \cite{Sed}, proving a more general formula, combining the result of differentiability of relative volume in \cite{BGM} and a method using Okunkov bodies. Our direct proof is independent of these techniques.

We'd also like to point out that a subclass of Shilov finite metrics (over a trivially valued field) has appeared in a recent study of n-A approach to K-stability for polarized varieties \cite{BJ2}: it consists of semipositive psh metrics whose Monge-Ampère measure is supported on a finite set of \emph{divisorial} points. It is shown in \emph{loc.cit.} that it suffices to test K-stability of a polarized variety via the $\beta$-invariant only through `generalized test configurations' that corresponds to such metrics.

\subsubsection{Another approach for global HS formula} Over a number ring, the equidistribution limit is calculated in \cite{Ni}, relying on a study of deformation to the normal cone construction for arithmetic line bundles, combined with normed Newton-Okunkov body techniques. The normed extension property for sections has been established earlier in \cite{Cha}. They directly work in a global setting, and their results also combine to give the arithmetic HS formula. We remark that both our approach over a local base and \emph{loc.~cit.} over a global base consider section algebras equipped with norms, and they are all inspired by the application of Grauert's construction to Arakelov geometry from \cite{Bos}.

\subsection{Conventions and notations}
Throughout the article, for simplicity we assume $\abs{k^{\times}}=\bbR_+$ and $k$ maximally complete, hence all norms are strict and with orthonormal basis; all norms are ultrametric; all Banach algebras (in calligraphic symbols such as $\KA$, $\KB$, $\KR$) are uniform, namely equipped with the spectral sup norm; they are usually obtained as the completion of some $k$-algebra (in Roman symbols such as $A$, $B$, $R$) with respect to an algebra norm $\algnorm{\ndot}$ (i.e. sub-multiplicative norm) which is power-multiplicative. As the line bundle $L$ is fixed, we'll sometimes omit it in the notations for section spaces and the section algebra, as $R_n$ and $R$.

\section{Banach spaces and Banach algebras}

\subsection{Banach spaces}

Recall that for a normed $k$-vector space $(E,\norm{\ndot})$, the closed and open unit balls are denoted by $E^{\circ}$ and $E^{\circ\circ}$, the unit sphere is denoted by $E^{\circ\setminus\circ\circ}$, the reduction $\widetilde{k}$-vector space is denoted by $\widetilde{E}$ (or $[E]^{\sim}$). For any subset $A$ of $E^{\circ}$, denote by $\widetilde{A}$ (or $[A]^{\sim}$) the subset $A/(A\cap E^{\circ\circ})$ of $\widetilde{E}$.

Recall that a finite set of vectors $\{e_i\}_{i\in I}$ in $E$ is \emph{orthogonal} (resp.\emph{orthonormal}) if $\norm{\sum_{i\in I} a_ie_i}=\max_{i\in I}\{\abs{a_i}\norm{e_i}\}$ for any coefficient $a_i\in k$ (resp. and $\norm{e_i}=1$) for any $i$. A finite set of vectors $\{e_i\}$ with $\norm{e_i}=1$ is orthonormal if and only if $\{\widetilde{e_i}\}_{i\in I}$ is $\widetilde{k}$-linearly independent. Since $(k,\abs{\ndot})$ is assumed to be spherically complete, any ultrametric Banach space $\KE$ is Cartesian; in particular, orthogonal projection exists for any closed subspace; since $\abs{k^{\times}}=\bbR_+$, any closed subspace possesses an orthonormal basis. (\cite[\S1-2]{BGR})

Recall that the complete tensor product $\KE\widehat{\otimes}\KF$ of two Banach spaces over a complete ultrametric valued field $(k,\abs{\ndot})$ (Banach field) is the completion of $\KE\otimes \KF$ with respect to the norm $\inf\max\{\norm{\ndot}_{\KE}\cdot\norm{\ndot}_{\KF}\}$ running over all possible presentations of a tensor ($\inf$) and over all its components ($\max$). Let $(l,\abs{\ndot}_l)$ be another Banach field, extension of $(k,\abs{\ndot})$. Then the operator norm of any linear operator $T: \KE\to \KF$ remains unchanged after base change to $T\otimes 1: \KE\widehat{\otimes}l\to \KF\widehat{\otimes}l$ (\cite{BGR}).

\begin{proposition}\label{P: VF extension operator norm}
Let $(k,\abs{\ndot})$ be a n-A Banach field and $(l,\abs{\ndot}_l)$ be a valued extension of it. Let $E$ be a $k$-linear subspace of $l$, of finite dimension $d$. Let $s\in l^{\times}$ be a non-zero element. Consider the linear map $(\cdot s)|_E$ from $E$ to $sE$, equip these two vector spaces with norms induced by $\abs{\ndot}_l$. Then the operator norm of $\bigwedge^d (\cdot s)|_E$ is equal to $(\abs{s}_l)^d$. 
\end{proposition}
\begin{proof}
    By rescaling one may assume $\abs{s}_l=1$, thus $\widetilde{s}\in \widetilde{l}$ is a non-zero element, hence the linear map of multiplication by $\widetilde{s}$ in the $\widetilde{k}$-vector space $\widetilde{l}$ is injective.
    
    Take an orthonormal basis $\{e_i\}_{i\in I}$ for $(E,\abs{\ndot}_l)$, then $\{\widetilde{e_i}\}_{i\in I}$ are linearly independent over $\widetilde{k}$; as $\widetilde{s\cdot e_i}=\widetilde{s}\cdot\widetilde{e_i}$ for any $i$ by the multiplicativity of $\abs{\ndot}_l$, the vectors $\{\widetilde{s\cdot e_i}\}_{i\in I}$ are also linearly independent over $\widetilde{k}$ by the injectivity of the map $(\cdot\widetilde{s})$, thus $\{s\cdot e_i\}_{i\in I}$ are orthonormal. Hence thanks to orthogonality, one calculates
    \[\norm{\bigwedge^d (\cdot s)|_E}_{\hom}=\norm{\bigwedge_{i\in I}s\cdot e_i}/\norm{\bigwedge_{i\in I}e_i}=\prod\norm{s\cdot e_i}/\prod\norm{e_i}=1=\abs{s}_l^d,\]
    getting the desired equality.
    Otherwise, one can also proceed by base change: the operator norm of $\bigwedge^d(\cdot s)|_E$ remains the same after analytic base change, equal to the operator norm of $\bigwedge^d [(\cdot s)\otimes 1]|_{E\widehat{\otimes}_kl}: E\widehat{\otimes}_k l \to sE\widehat{\otimes}_kl$, where the multplication operator becomes a scalar multiplication as $s\otimes 1$ can be identified with the scalar $1\otimes s$. This new operator norm is $(\abs{s}_l)^d$.
\end{proof}

\begin{construction}
    Let $E$ be a vector space over $k$, and $\{\norm{\ndot}_i\}_{i\in I}$ be a family of ultrametric norms on $E$. Denote by $\KE_i$ the Banach space of completion of $E$ with respect to $\norm{\ndot}_i$, and by $\iota_i: E\to \KE_i$ and $\iota: E\mapsto\prod_{i}E_i$ the canonical embeddings. The linear subspace $\prod'\KE_i$ of $\prod\KE_i$ consisting of $\prod v_i$ with $\sup_{i\in I}\norm{v_i}_i<\infty$ is a Banach space denoted by $\obot\KE_i$. Then $\sup_{i\in I}\norm{\ndot}_i$ defines a norm, written as $\norm{\ndot}_I$, on the linear subspace $(\prod\iota_i)^{-1}\prod'\KE_i$ of $E$, denoted by $E'_I$, whose unit ball is $E'\cap \bigcap_i \iota_i^{-1}\KE_i^{\circ}$. Denote by $\KE_I$ the Banach space of completion of $(E',\sup_{i\in I}\norm{\cdot}_i)$. 
\end{construction}

\begin{construction}\label{Cs: contraction subspace and map}
    Let $l: E\to F$ be a linear map of normed $k$-vector spaces; suppose $f$ is \emph{contractive} in the sense that $\norm{l(\ndot)}_F\leq\norm{\ndot}_E$. The \emph{contraction of $l(E)$ in $F$} is a linear subspace $F'\subset F$ (\emph{contraction subspace}) together with a contractive linear map $\kappa: l(E)\to F'$ (\emph{contraction map}) such that $\widetilde{\kappa}: [l(E^{\circ})]^{\sim}\rightarrow [F']^{\sim}$ is an isomorphism of $\widetilde{k}$-vector spaces. 
    
    They can be constructed as follows: choose $\widetilde{k}$-basis $\{\widetilde{e_j}\}_{j\in J}$ of $[l(E^{\circ})]^{\sim}$, and their lifts $\{e_j\}$ in $F$, which span $F'$; choose complementary $k$-basis $\{f_i\}_{i\in I}$ that is orthonormal\footnote{here we use the assumption that $k$ is maximally complete} to $\{e_j\}$ in $l(E^{\circ\setminus\circ\circ})\cap F^{\circ\circ}$ for $F'\subset l(E)$, and $\kappa$ is defined by the orthogonal projection $e_j\mapsto \widetilde{e_j}$ and $f_i\mapsto 0$. Note that both $F'$ and $\kappa$ depend on the choice of various $k$-basis, but $\dim_k F'$ is always equal to $\dim_{\widetilde{k}}[l(E^{\circ})]^{\sim}$ independent of choices by construction.
\end{construction}
\begin{proposition}\label{P: join contraction is injective and isometric}
    Let $E$ be a vector space over $k$, and $\{\norm{\ndot}_i\}_{i\in I}$ be a family of ultrametric norms on $E$. Denote by $\KC_i$ a contraction of $\pi_i\iota(\KE_{I})$ in $\KE_i$ with contraction map $\kappa_i$. Then $\kappa=\prod\kappa_i: \KE_{I}\rightarrow\obot_{i\in I}\KC_i$ is injective and isometric.
\end{proposition}
\begin{proof}
    For any $e\in\KE_{I}^{\circ\setminus\circ\circ}$, there exists $i$ for which $\pi_i\iota(e)$ is in $\KE_i^{\circ\setminus\circ\circ}$, hence $\kappa_i\pi_i\iota(e)$ is in $\KC_i^{\circ\setminus\circ\circ}$ by the construction of contraction. Hence $\KE_{I}^{\circ\setminus\circ\circ}$ is mapped by $\kappa$ into $[\obot_{i\in I}\KC_i]^{\circ\setminus\circ\circ}$ which implies the injectivity and isometricity of $\kappa$.
\end{proof}

Recall that the operator norm $\norm{S}_{\hom(E,F)}$ of a linear map $S: E\to F$ is defined by $\sup_{0\neq v\in E}\norm{Sv}_F/\norm{v}$, and the distance between two norms $\dist(\norm{\ndot}_1,\norm{\ndot}_2)$ on $E$ is $\sup_{0\neq v\in E}\big|\abs{\log\abs{\norm{v}_1/\norm{v}_2}} \big|$. The following two esitmates are elementary:
\begin{proposition}\label{P: det norm distorsion subspace}
    Let $E$ and $E'$ be finite dimensional Banach spaces and let $S\in L(E, E')$ be an invertible linear operator; let $F$ and $F'=S(F)$ be subspaces and $m\in\bbN$ be the codimension $\dim E-\dim F$. Let $C\in\bbR_+$ be a constant such that $C^{-1}<\max\{\norm{S}_{\hom},\norm{S^{-1}}_{\hom}\}<C$. Then it holds that
    \[ C^{-m}\norm{\det S|_E}_{\hom}\leq \norm{\det S|_F}_{\hom}\leq C^m \norm{\det S|_E}_{\hom}.\]
\end{proposition}
\begin{proposition}\label{P: operator norm changes}
    Let $E$ and $F$ be $k$-vector spaces and let $S\in L(E, E')$ be a linear operator. Let $\norm{\ndot}_{E,i}$ and $\norm{\ndot}_{F,i}$ be norms on $E$ and $F$ for $i=1,2$; write $E_i, F_i$ for the corresponding normed vector spaces, then
    \[ \big|\log\abs{ \norm{S}_{\hom(E_1,F_1)}/\norm{S}_{\hom(E_2,F_2)}}\big|\leq \dist(\norm{\ndot}_{E,1},\norm{\ndot}_{E,2})+\dist(\norm{\ndot}_{F,1}, \norm{\ndot}_{F,2}).\]
\end{proposition}

\begin{proposition}\label{P: first singular value wrt max}
    Let $E$ be a finite dimensional vector space and let $S\in L(E, E)$ be a linear operator. For any norm $\norm{\ndot}$ on E, denote by $\sigma(S,\norm{\ndot})$ the operator norm of $S$ with respect to $\norm{\ndot}$. Let $\norm{\ndot}_1$, $\norm{\ndot}_2$ be two norms, then
    \[\sigma(S,\max\{\norm{\ndot}_1, \norm{\ndot}_2\})\leq \max\{\sigma(S,\norm{\ndot}_1), \sigma(S,\norm{\ndot}_2)\}.\]
\end{proposition}
\begin{proof}
    Since $(k,\abs{\ndot})$ is spherically complete and $E$ is finite dimensional, one can find a basis $\{e_i\}_{i\in I}$ that is orthogonal for both ultrametric norms $\norm{\ndot}_1$ and $\norm{\ndot}_2$. Write $\norm{\ndot}_{\max}=\max\{\norm{\ndot}_1, \norm{\ndot}_2\}$, it is the norm having $\{e_i\}_{i\in I}$ as orthogonal basis with $\norm{e_i}_{\max}=\max\{\norm{e_i}_1,\norm{e_i}_2\}$. Denote by $I_1$ the subset of $I$ for which $\norm{e_i}_{1}\geq \norm{e_i}_{2}$ and by $I_2$ its complement, and by $E_1$ the subspace generated by basis vectors in $I_1$ and similarly $E_2$, thus $E=E_1\oplus E_2$ is an orthogonal direct sum for all three norms.

    For any $v\in E\setminus\{0\}$ with $\norm{v}_{\max}=1$, one can write $v=v_1+v_2$ according to the decomposition; thanks to the ultrametric inequality, the desired estimate follows from
    \[\norm{Sv}_{\max}=\norm{Sv_1+Sv_2}_{\max}\leq\max\{\norm{Sv_1}_1,\norm{Sv_1}_2, \norm{Sv_2}_1, \norm{Sv_2}_2\}\leq\]
    \[\max\big\{\sigma(S,\norm{\ndot}_1)\max\{\norm{v_1}_1, \norm{v_2}_1\}, \sigma(S,\norm{\ndot}_2)\max\{\norm{v_1}_2, \norm{v_2}_2\}\big\}=\max\{\sigma(S,\norm{\ndot}_1), \sigma(S,\norm{\ndot}_2)\}.\]
\end{proof}

\subsection{Shilov boundary}
We recall some notions and properties concerning boundaries of a uniform (i.e. spectral) Banach algebra $\KB$. Recall that its \emph{spectrum} $\KM(\KB)$ consists of semi-absolute values $\abs{\ndot}_z$ (also sometimes denoted by $\abs{(\ndot)(z)}$) bounded by the spectral norm $\norm{\ndot}_{\Sp}$, equipped with the Gelfand-Berkovich topology. A point $z$ is \emph{birational} if the ideal $\ker(\abs{\ndot}_z)$ is $\{0\}$. 

\begin{definition}\label{D: boundaries}
A \emph{boundary} for a uniform Banach algebra $\KB$ is a subset $\Delta$ of $\KM(\KB)$, such that $\norm{f}_{\Sp}=\sup_{z\in \Delta}\abs{f}_z$ holds for any $f\in\KB$. By construction, one can decompose the sup norm $\algnorm{\ndot}_{\KB}$ as $\sup_{a\in\Delta}\abs{\ndot}_a$. A \emph{minimal boundary} is a boundary that is minimal for the inclusion partial order (denoted by $\partial_0\KB$). A \emph{Shilov boundary} is a \emph{closed} boundary minimal for inclusion among the set of all closed boundaries (denoted by $\Shbd\KB$).
\end{definition}
\begin{remark}
    It is elementary that Shilov boundary exists thanks to Zorn's lemma: the intersection $\Gamma_{\infty}$ of a strictly decreasing chain of closed boundaries $\{\Gamma_i\}_{i\in I}$ is a lower bound: for any $f\in \KB^{\circ\setminus\circ\circ}$, there exists $z_i\in P(\abs{f})\cap \Gamma_i$ for any $i\in I$, and the $I$-net $\{z_i\}_{i\in I}$ has a cluster point $z_{\infty}$ in $\KM(\KB)$ by compactness, which is contained in the closed set $\Gamma_{\infty}$ by monotonicity (for non-closed boundaries, this cluster point may not be contained in the intersection). By continuity $z_{\infty}\in P(\abs{f})$, so $\Gamma_{\infty}$ is a closed boundary. It is not clear whether minimal boundary exists.
    
    \emph{A priori}, Shilov boundary is not necessarily unique. A miminal boundary (if exists) is not necessarily closed nor unique; its topological closure is a Shilov boundary (the one associated to this minimal boundary). Note that these are definitions in \cite{EM} (in accordance with the general theory of function algebras) instead of those given in \cite[\S 2.4]{Ber}. 
\end{remark}

\begin{example}\label{Ex: envelope}
    Let $S$ be a subset in $\KM(\KB)$. Its \emph{holomorphic envelope}, denoted by $\KP(S)$, is the subset consisting of points $z$ such that $\abs{f(z)}\leq \sup_{w\in S}\abs{f(w)}$ for all $f\in\KB$. It is a closed (hence compact) subset of $\KM(\KR)$. By construction, it holds that $\Shbd\KP(S)\subseteq S$ if $S$ is closed. One can similarly define holomorphic envelope for any subset in $(\spec B)^{\an}$ using functions $f\in B$ for any $k$-algebra $B$. 

    Given two points $z,w$ in $\KM(\KB)$ (or in $(\spec B)^{\an}$), define an ordering by: $z\prec_{\KP}w$ if and only if $\abs{f(z)}\leq \abs{f(w)}$ for any $f\in \KB$ (or $f\in B$); by construction, the last condition is equivalent to the condition $\{z\}\subseteq \KP(\{w\})$.  
\end{example}

\begin{construction}\label{Cs: reduction an element}
    Recall that the reduction map $\red: \KM(\KB)\rightarrow \spec(\widetilde{\KB})$ is the one induced by $[\KB\xrightarrow{z}\SH(z)]\mapsto \ker[\widetilde{\KB}\xrightarrow{\widetilde{z}}\widetilde{\SH(z)}]$. Denote by $\eta(\widetilde{\KB})$ the set of generic points for its irreducible components, namely the minimal prime ideals of $\widetilde{\KB}$. For any $f\in\KB^{\circ\setminus\circ\circ}$, denote by $P(\abs{f})$ the closed subset of $\KM(\KB)$ where $\abs{f(\ndot)}=1$; via norm reduction, it is equal to $\red^{-1}D(\widetilde{f})$ where $D(-)$ denotes the principal Zariski open set of $\spec\widetilde{\KB}$. By construction, $z\in P(\abs{f})$ if and only if $\red(z)\in D(\widetilde{f})$. Moreover, it holds that $P(\abs{f})\cap P(\abs{g})=P(\abs{fg})$ if the intersection is not empty.
\end{construction}

\begin{proposition}\label{P: no generisation in minimal boundary}
    Let $\KB$ be a uniform Banach algebra, then
    \begin{enumerate}
        \item if $\frak{p}$, $\frak{q}$ are points in $\spec\widetilde{\KB}$, with $\frak{p}\subset \frak{q}$, then for any $z, w$ with $\red(z)=\frak{p}$ and $\red(w)=\frak{q}$, and any $f\in\KB^{\circ}$, it holds that $w\in P(\abs{f})$ implies $z\in P(\abs{f})$;
        \item if $\Delta$ is a boundary, $w\in\Delta$, and $z$ is a point such that $\red(z)\subset\red(w)$, then $(\Delta\setminus\{w\})\cup\{z\}$ is still a boundary; 
        \item if $\KB$ has a minimal boundary $\partial_0\KB$, then there is no generisation relation between any two points in $\red(\partial_0\KB)$.
    \end{enumerate}
\end{proposition}
\begin{proof}
    The (i) is a consequence of the fact that the Zariski open set $D(\widetilde{f})$ is closed under generisation from $\frak{q}$ to $\frak{p}$. The (ii) and (iii) are immediate consequences.
\end{proof}

\begin{proposition}\label{P: element and reduction}
    Let $\KB$ be a uniform Banach algebra, and $\frak{p}\in (\spec\widetilde{\KB})^{\an}$ be a prime ideal, then
    \begin{enumerate}
        \item if $f,g\in \KB^{\circ\setminus\circ\circ}$ are elements such that neither $\widetilde{f}$ nor $\widetilde{g}$ is in $\frak{p}$, then $P(\abs{f})\cap P(\abs{g})\neq\emptyset$; 
        \item if $\frak{p}\in \eta(\widetilde{\KB})$, then $\red^{-1}(\frak{p})=\bigcap_{\widetilde{f}\notin\frak{p}}P(\abs{f})$ is a non-empty closed set; consequently $\red(\KB)$ contains $\eta(\widetilde{\KB})$;
        \item the set $\red^{-1}(\eta(\widetilde{\KB}))=\bigcup_{\frak{p}\in \eta(\widetilde{\KB})}\red^{-1}(\frak{p})$ is a boundary;
        \item any tuple of points in $\prod_{\frak{p}\in \eta(\widetilde{\KB})}\red^{-1}(\frak{p})$, viewed as a subset of $\KM(\KB)$, is a boundary.
    \end{enumerate}
\end{proposition}
\begin{proof}
    For (1), by construction the product $\widetilde{f}\cdot\widetilde{g}$ is not in $\frak{p}$, hence is non-zero; thus it has a lift $h\in \KB^{\circ\setminus\circ\circ}$. As $D(\widetilde{f})\cap D(\widetilde{g})=D(\widetilde{h})$, and $\red^{-1}D(\widetilde{h})=P(\abs{h})$ is non-empty, hence so is $\red^{-1}(D(\widetilde{f})\cap D(\widetilde{g}))=P(\abs{f})\cap P(\abs{g})$.

    For (2), as $\frak{p}$ is a minimal prime ideal, one can write $\frak{p}$ as $\bigcap_{\widetilde{f}\notin\frak{p}}D(\widetilde{f})$, hence $\red^{-1}$ is the intersection of closed subsets $\{P(\abs{f})\}_{\widetilde{f}\notin\frak{p}}$. By (i), this collection of subsets has finite intersection property. Thus by compactness of $\KM(\KB)$, this intersection is a non-empty closed subset, consequently $\eta(\widetilde{\KB})\subseteq \red(\KB)$.

    For (3), for any $f\in \KB^{\circ\setminus\circ\circ}$, as $\widetilde{f}\neq 0$, there exists $\frak{p}\in \eta(\widetilde{\KB})$ such that $\widetilde{f}(\frak{p})\neq 0$, thus $\abs{f}$ takes maximum on $\red^{-1}(\frak{p})$ (non-empty by (2)). This shows that $\red^{-1}(\eta(\widetilde{\KB}))$ is a boundary.

    For (4), it follows from (3) immediately.
\end{proof}
\begin{remark}
    Note that apart from the case where $\KB$ is a (strict) affinoid algebra, it is not clear whether $\red$ is surjective. 
\end{remark}

\begin{definition}\label{D: peak-isolated point}
    A point $z\in \KM(\KB)$ is \emph{peak-isolated} if $\bigcap_{z\in P(\abs{f})}P(\abs{f})=\{z\}$.
\end{definition}

\begin{proposition}\label{P: peak point if red saturated}
    Let $\KB$ be a uniform Banach algebra, and $z\in \KM(\KB)$ be a point which is peak-isolated, then
    \begin{enumerate}
        \item for any (open) neighbourhood $V$ of $z$, one can find $f_V\in\KB^{\circ\setminus\circ\circ}$ such that $P(\abs{f_V})\subseteq V$;
        \item the point $z$ lies in any Shilov boundary $\Shbd\KB$.
    \end{enumerate}
\end{proposition}
\begin{proof}
    For (1), note that the collection of closed subsets in the intersection has finite intersection property. As $V^{\complement}\cap \bigcap_{z\in P(\abs{f})}P(\abs{f})=\emptyset$, by compactness there are a finite number of $f$'s that makes such an intersection still empty. The product of these $f$'s can be taken as $f_V$.

    For (2), by construction, for any Shilov boundary $\Shbd\KB$, each $P(\abs{f})$ should contain a point in it, thus there is a net of points in $\Shbd\KB$ converging to $z$, hence $z$ is in this Shilov boundary as it is a closed set. 
\end{proof}

It is however shown in \cite{Gue} \cite{EM} that there is a way to construct a particular Shilov boundary for any uniform Banach algebra; see Appendix A for more details. In the main part of this article, we shall only deal with a simple case which suffices for our use, namely when one (hence all, by Proposition \ref{P: finite Shilov boundary pic isolated}) $\Shbd\KB$ is finite, and will not use this construction.

\begin{proposition}\label{P: finite Shilov boundary basic}
    Let $\KB$ be a uniform Banach algebra. Suppose that it has a Shilov boundary $\Shbd\KB$ which is a finite set $\{\xi_i\}_{i\in I}$, then
    \begin{enumerate}
        \item this Shilov boundary is a minimal boundary; 
        \item for any proper subset $J\subset I$, there exists $f\in\KB^{\circ}$, such that $\abs{f(\xi_i)}<1$ for all $i\in J$ and $\abs{f(\xi_j)}=1$ for all $j\in I\setminus J$. If $\KB$ is graded, then $f$ can be chosen to be homogeneous;
        \item every $\frak{p}_i:=\red(\xi_i)$ is a minimal prime, the map $\red$ restricts to a bijection $\Shbd\KB\to \eta(\widetilde{\KB})$; 
        \item the set of Shilov boundaries of $\KB$ is in bijection with points (viewed as a subset of $\KM(\KB)$) in $\prod_{\Shbd\KB}\red^{-1}\red(\xi_i)$ ;
        
    \end{enumerate}
\end{proposition}
\begin{proof}
    For (1), since $\Shbd\KB$ is a finite set, any subset of it is closed, hence itself is a minimal boundary by definition.

    For (2), it suffices to find an element $\widetilde{f}\in\widetilde{\KB}$ which lies in $\frak{p}_i$ for $i\in J$ but not in $\frak{p}_i$ for $i\in I\setminus J$. Since $\Shbd\KB$ is a minimal boundary, there is no generisation relations, thus $\frak{p}_j\setminus\frak{p}_i$ is non-empty for any $i\notin J$ and $j\in J$. By taking product $\bigcap_{j\in J}\frak{p}_j\setminus \frak{p}_i$ is non-empty for any $i\notin J$, by prime avoidance $\bigcap_{j\in J}\frak{p}_j\setminus \bigcup_{i\notin J}\frak{p}_i$ is non-empty.
    
    For (3), first show that $\frak{p}_i$ are minimal and $\red$ is surjective to $\eta(\widetilde{\KB})$: otherwise there exists $\frak{p}\in\eta(\widetilde{\KB})\setminus \frak{p}_i$, then any $\frak{p}_i\setminus\frak{p}$ is non-empty, hence so is $\cap\frak{p}_i\setminus\frak{p}$. Choose an element $\widetilde{f}$ and a lift $f\in\KB^{\circ}$, it holds that $\abs{f(\xi_i)}<1$ for all $i$ while $\norm{f}_{\Sp}=1$, a contradiction to the definition of Shilov boundary. The injectivity follows from the minimality of boundary. The (4) is an immediate consequence.
\end{proof}

The following statements about analytic localization and reduction are variants of \cite[\S 7.2.6]{BGR} and \cite[Proposition 2.4.4 (ii)]{Ber} with similar arguments.
\begin{lemma}\label{L: norm reduction localization Banach alg finite Shilov}
    Let $\KB$ be a uniform Banach algebra. Let $f\in\KB^{\circ\setminus\circ\circ}$, consider the Banach algebra homomorphism $\sigma: \KB\to \KB\{f^{-1}\}$. Denote by $\norm{\ndot}_{1}$ and $\norm{\ndot}_{2}$ the spectral (semi-)norms on $\KB$ and $\KB\{f^{-1}\}$. Then
    \begin{enumerate}
        \item the subspace $\KM(\KB\{f^{-1}\})$ is identified with the subset $P(\abs{f})$ in $\KM(\KB)$;
        \item for any $b\in\KB$, it holds that $\lim_{n\to\infty}\norm{f^n\cdot b}_1=\norm{\sigma(b)}_2$;
        \item the homomorphism $\widetilde{\KB}[\widetilde{f}^{-1}]\to \widetilde{\KB\{f^{-1}\}}$ is injective; if there is a Shilov boundary $\Shbd\KB$ which is a finite set, then it is also surjective.
    \end{enumerate}
\end{lemma}
\begin{proof}
    For (1), it is elementary: the defining homomorphism $\KB\to\KB\{f^{-1}\}$ (where the second term is just $\KB\{ T\}/(Tf-1)$) induces a continuous map $\KM(\KB\{f^{-1}\})\to \KM(\KB)$, the fiber over a point $z$ is the spectrum of $\SH(z)\{ T\}/(Tf(z)-1)$ which is non-empty exactly when $\abs{f(z)}\geq 1$.
    
    For (2), since $\abs{f(\ndot)}=1$ on $\KM(\KB\{f^{-1}\})$, and $\abs{f(\ndot)}<1$ on its complement by $(1)$, the formula holds. Note that $\{\norm{f^n\cdot b}_1\}_{n\in\bbN}$ is a decreasing sequence.

    For (3), for the injectivity: since $\widetilde{1}=\widetilde{f\cdot f^{-1}}=\widetilde{f}\cdot\widetilde{f^{-1}}$, the element $\widetilde{f}$ is invertible in $\widetilde{\KB\{f^{-1}\}}$, thus there is an induced homomorphism $\iota$ from $\widetilde{\KB}[\widetilde{f}^{-1}]$ by universal property. Then the natural homomorphism $\widetilde{\sigma}: \widetilde{\KB}\to \widetilde{\KB\{f^{-1}\}}$ is the composition of the localization $\delta: \widetilde{\KB}\to \widetilde{\KB}[\widetilde{f}^{-1}]$ with $\iota$, and it suffices to show the injectivity of $\widetilde{\sigma}$. Let $b\in\KB^{\circ}$ be an element such that $\widetilde{\sigma}(\widetilde{b})=0$, namely $\widetilde{\sigma(b)}=0$; thanks to the identity in (2), one knows $\lim\norm{f^nb}_1=\norm{\sigma(b)}_2<1$, thus there exists $n\in\bbN$ such that $\norm{f^nb}_1<1$, hence $0=\widetilde{f^n b}=\widetilde{f}^n\widetilde{b}$. This means that $\delta(\widetilde{b})$ is $0$ in $\widetilde{\KB}[\widetilde{f}^{-1}]$. 
    
    For the surjectivity under finiteness assumption, for any $c\in \KB[f^{-1}]$ with $\norm{c}_2=1$, write $c=f^{-j}b$ with $b\in\KB$, then $\abs{b(\ndot)}=1$ on $P(\abs{f})$. As $\Shbd\KB$ is a finite set $\{\xi_i\}_{i\in I}$, it holds that the sup norm is a max of finitely many absolute values $\abs{\ndot(\xi_i)}$, hence $\norm{f^n b}_1=\max_{i\in I}\{\abs{f(\xi_i)}^n\abs{b(\xi_i)}\}$, and the set of such values for all $n\in\bbN$ is a subset of $\bbR$ with at most one accumulation point equal to $0$. From this equality and the discreteness, one can find $n\in\bbN$ such that $\norm{f^n b}\leq 1$. Thus $c=f^{-j-n}(f^nb)$ and hence $\widetilde{f^nb}(\widetilde{f})^{-j-n}$ maps to $\widetilde{c}$. As $\KB[f^{-1}]$ is dense in $\KB\{f^{-1}\}$, surjectivity holds.
\end{proof}
\begin{proposition}\label{P: finite Shilov boundary pic isolated}
    Let $\KB$ be a uniform Banach algebra. Suppose that it has a Shilov boundary $\Shbd\KB$ which is a finite set, then
    \begin{enumerate}
        \item any $\xi\in \Shbd\KB$ is peak-isolated;
        \item both minimal and Shilov boundaries are unique, and $\partial_{0}\KB=\Shbd\KB$. 
    \end{enumerate}
        
\end{proposition}
\begin{proof}
    For (1), if there is only one $\xi$ point in a Shilov boundary, then $\red^{-1}\red{\xi}$ is the singleton $\{\xi\}$: otherwise, let $\xi'$ be a different point in it, one can find $f\in \KB^{\circ\setminus\circ\circ}$ such that $\abs{f}$ separates $\xi$ and $\xi'$; but this contradicts $\widetilde{f}\notin\red(\xi)$.

    In general, let $\xi$ be a point in a Shilov boundary $\Shbd\KB$, with reduction $\frak{p}$. One can find $f\in \KB^{\circ\setminus\circ\circ}$ such that $P(\abs{f})$ contains $\xi_i$ but not any other points in $\Shbd\KB$. By Lemma \ref{L: norm reduction localization Banach alg finite Shilov}(3), one has $\widetilde{\KB\{f^{-1}\}}=\widetilde{\KB}[\widetilde{f}^{-1}]$; thus there is only one point $\frak{p}$ in $\eta(\widetilde{\KB\{f^{-1}\}})$. As $\red^{-1}(\frak{p})$ is contained in $P(\abs{f})$, this inverse image can be computed using the $\red$ map on $\KB\{f^{-1}\}$ by Lemma \ref{L: norm reduction localization Banach alg finite Shilov}(1), and it is only one point by (1).

    The (2) follows from (1) and Proposition \ref{P: finite Shilov boundary basic}(2).
\end{proof}


\begin{construction}\label{Cs: Banach algebra decomposition}
    Let $B$ be a $k$-algebra, and $\{\abs{\ndot}_a\}_{a\in I}$ be a family of absolute values on $B$. For any non-empty $J\subset I$, denote by $\KB_J$ the Banach algebra of completion of the normed $k$-algebra $(B,\sup_{a\in J}\abs{\ndot}_a)$. The canonical embedding $\iota:\KB_{I}\rightarrow\obot_{a\in A}\KB_{a}$ is an injective and isometric homomorphism of uniform Banach algebras. It induces an injective homomorphism of $\widetilde{k}$-algebras after reduction: $\widetilde{\iota}: \widetilde{\KB_{I}}\rightarrow \prod_{a\in I}\widetilde{\KB_{a}}$.
\end{construction}

\section{Equidistribution measure}
In this part we define an equidistribution measure and calculate the determinant norm distorsion for semipositive metrics.

\subsection{Normed section algebra and disc domain}
We recall some basic constructions from our previous work \cite{Fan1} encoding a (possibly singular) semipositive metric in a Banach algebra norm (see also \cite{BJ1} and \cite{Reb1} for relevant constructions). A metric is \emph{semipositive} if it is a pointwise limit of non-archimedean Fubini-Study metrics; semipositive metrics are always lower semicontinuous in the sense that the function $x\mapsto \abs{s(x)}_{\phi}$ is upper semicontinuous (for any section $s$); it is possibly singular in the sense that it could be discontinuous or degenerate at some point.
\begin{construction}
    Let $X$ be an irreducible projective variety with $L$ an ample line bundle over it. Denote by $R(L)$ the graded algebra of sections, and by $\Cone(L)$ the affine cone associated to $(X,L)$, which is just $\spec R(L)$.
    
    Let $\phi$ be a metric on $L$. Denote by $\norm{\ndot}_{n\phi}$ the sup norms induced by $\phi$ on $R_n(L)$ and by $\algnorm{\ndot}_{\phi}=\sup_{n\in\bbN}\norm{\ndot}_{n\phi}$ the graded algebra norm on $R(L)$. Denote by $\KR(L,\phi)$ the \emph{normed section algebra}, which is the Banach algebra completion of $R(L)$ with respect to $\algnorm{\ndot}_{\phi}$. Its spectrum $\KM(\KR(L,\phi))$ is the \emph{disc domain} of $(X,L,\phi)$ (the unit disc bundle of $(L,\phi)^{\vee}$ with zero section contracted). Conversely, any graded algebra norm $\algnorm{\ndot}$ on $R(L)$ defines a (possibly singular) semipositive metric $\KP(\algnorm{\ndot})$ as the limit of $\frac{1}{n}\FS(\norm{\ndot}_{n})$. It is a (possibly discontinuous but non-degenerate) metric if and only if $\algnorm{\ndot}$ is bounded below by $\algnorm{\ndot}_{\psi}$ for some continuous metric $\psi$, or equivalently from the geometric point of view, the associated disc domain contains a neighbourhood of $\pmb{0}$.
\end{construction}

\begin{construction}
    Denote by $p:\V(L)\to X$ the tautological projection from the total space of $L^{\vee}$ and by $\KM(\phi)$ the image under $p^{\an}$ of the dual unit disc bundle $\overline{\bbD}^{\vee}(L,\phi)$ of $\phi$. The envelope metric $\KP(\phi)$ on $L$ is the limit of Fubini-Study metrics $\frac{1}{n}\FS(\norm{\ndot}_{n\phi})$. One can show that $\algnorm{\ndot}_{\KP(\phi)}$ coincides with $\algnorm{\ndot}_{\phi}$ hence the corresponding Banach algebras and spectra are the same. Geometrically, the envelope metric can also be characterised by identifying the holomorphic envelope $\KP(\KM(\phi))$ with its disc domain $\KM(\KP(\phi))=\KM(\KR(L,\phi))$. For any $x\in X^{\an}$, denote by $\gamma_{\phi}(x)\in\KM(\phi)$ the Gauss point of the dual unit disc over $x$; view $s$ as a function on $\V(L)$, one has
    $\abs{s(x)}_{\phi}=\abs{s(\gamma(x,\phi))}$.
\end{construction}

\begin{construction}\label{Cs: gradization}
    A norm $\norm{\ndot}$ on $R_{\sbullet}$ is \emph{graded} if $\norm{\cdot}=\norm{\cdot}^+$, where the later is the norm defined by $\sup_{n\in\bbN}\norm{(\cdot)_n}$ (whenever it's finite); for any point $z\in\KM(\KR)$, denote by $z^+$ the point corresponding to the absolute value $(\abs{\cdot}_z)^+$, which is still a point of $\KM(\KR,\phi)$ since $\algnorm{\ndot}_{\phi}$ is graded. By construction $\abs{\cdot}_z\leq \abs{\cdot}_{z^+}$, hence in any Shilov (or minimal) boundary, any point $z$ can be replaced by its graded version $z^+$.
\end{construction}

\begin{proposition}\label{P: algebra norm with prescribed Shilov boundary}
    Given any finite subset $\Sigma\subset \Cone(L)^{\an}_{\bir}$, one can construct a power-multiplicative algebra norm $\algnorm{\ndot}_{\Sigma}$ such that $\Shbd\KR(L,\algnorm{\ndot}_{\Sigma})$ contained in $\Sigma$. Moreover, the set $\KR(L,\algnorm{\ndot}_{\Sigma})$ coincide with the holomorphic envelope $\KP(\Sigma)$. If absolute values corresponding to points in $\Sigma$ are graded, so is the norm $\algnorm{\ndot}_{\Sigma}$.
\end{proposition}
\begin{proof}
    The norm $\max_{z\in\Sigma}\abs{\ndot}_z$ is power-multiplicative, and by construction, any element $f\in \KR(L,\algnorm{\ndot}_{\Sigma})$ takes maximal absolute value on $\Sigma$, hence there is a Shilov boundary $\Shbd\KR(L,\algnorm{\ndot}_{\Sigma})$ contained in this finite set. Thus the Shilov boundary is unique by Proposition \ref{P: finite Shilov boundary pic isolated}. This gives the first two assertions. The last assertion holds by construction.
\end{proof}

\subsection{Shilov finite line bundle metric}
\begin{definition}
    A metric $\phi$ on $L^{\an}$ is \emph{Shilov finite} if $\Shbd \KR(L,\phi)$ is a finite set of birational points.
\end{definition}
\begin{remark}
    This notion only depends on its induced uniform algebra norm on $R$, hence a metric $\phi$ is Shilov finite if and only if its envelope metric $\KP(\phi)$ is, because $\algnorm{\ndot}_{\phi}$ coincides with $\algnorm{\ndot}_{\KP(\phi)}$. A model metric $\phi_{\SX}$ (as well as its envelope) for some integral/formal model $\SX$ over $k^{\circ}$ is Shilov finite. More examples include locally affinoid metrics (to be discussed later).
\end{remark}

\begin{construction}\label{Cs: section algebra localize}
    For any $a\in \Shbd$, denote by $\KR(L)_a$ the graded Banach algebra completion of $(R(L),\abs{\ndot}_a)$. Consider the contractive homomoprhism of $k$-Banach algebras $\KR(L,\phi)\rightarrow \KR(L)_a$, denote by $\pi_a\KR(L,\phi)$ its image. By definition, elements in $[\pi_a\KR(L,\phi)]^{\sim}$ are classes in $\KR^{\circ}/(\KR^{\circ}\cap\KR_a^{\circ\circ})$, which are elements $\widetilde{s}\in [\KR(L,\phi)]^{\sim}$ such that $\widetilde{s}(\widetilde{a})\neq 0$; equivalently, they can be represented by $s\in\KR(L,\phi)$ such that $\algnorm{s}_{\phi}=1$ and $\abs{s}_a=1$, namely elements that takes maximum absolute value at $a$.
\end{construction}

\begin{construction}\label{Cs: section algebra decompose}
    Assume that $\Shbd \KR(L,\phi)$ is a finite set of birational points. The decomposition of sup norm $\algnorm{\ndot}_{\phi}=\sup_{a\in\Shbd}\abs{\ndot}_a$ implies that $\KR(L,\phi)=R(L)_{\Sh}$. From Construction \ref{Cs: Banach algebra decomposition}, one has universal embedding $\iota:  \KR(L,\phi)\rightarrow\obot_{a\in\Sh}\KR(L)_{a}$, and for any $a\in\Sh$ the Banach algebra $\pi_a\iota\big(\KR(L,\phi)\big)$ is just $\pi_a\KR(L,\phi)$. One also constructs an injective isometric homomorphism into a Banach $k$-subalgebra $\iota_-: \KR(L,\phi)\rightarrow\obot_{a\in\Sh}\pi_a\KR(L,\phi)$, as well as an injective homomorphism on reduction $\widetilde{k}$-algebras: $[\iota_-]^{\sim}: [\KR(L,\phi)]^{\sim}\rightarrow \prod_{a\in\Sh}[\pi_a\KR(L,\phi)]^{\sim}$. 
\end{construction}

\begin{proposition}\label{P: Shilov point of localized normed section algebra}
    For each $a\in\Shbd\KR(L,\phi)$, the Banach algebra $\KR(L)_a$ as well as $\pi_a\KR(L,\phi)$ has only one Shilov point, which is mappped to $a$ by the map on spectra induced by natural homomorphisms from $\KR(L,\phi)$.
\end{proposition}
\begin{proof}
    By construction, the sup norm on $\KM(\KR(L)_a)$ is $\abs{\ndot}_a$, hence $a$ is its unique Shilov point. It is mapped to $a$ in $\KM(\KR(L,\phi))$ by the map induced by $\KR(L,\phi)\rightarrow \KR(L)_a$. Similar argument works for $\KM(\pi_a\KR(L,\phi))$.
\end{proof}

\begin{proposition}\label{P: compression section algebra}
    Assume that $\phi$ is Shilov finite. For each $a\in\Sh$, one can find a contraction $k$-linear graded subspace $\KC_a$ of the subspace $\pi_a\KR(L,\phi)$ in $\KR(L)_a$ with a contraction map $\kappa_a: \pi_a\KR(L,\phi)\rightarrow \KC_a$. The induced linear map $\kappa: \KR(L,\phi)\rightarrow \prod_{a\in\Sh}\KC_a$ is graded and is injective and isometric.
    
\end{proposition}
\begin{proof}
    This follows from Construction \ref{Cs: contraction subspace and map}; note that one can choose corresponding basis $\{e\}$ and $\{f\}$ to be homogeneous elements so as to induce a grading on $\KC_a$ by $n\in\bbN$ from $\KR(L)_a$ and to make $\kappa$ respecting the grading.    
\end{proof}

\begin{proposition}\label{P: inverse of multiplication bounded}
    Assume that $\phi$ is Shilov finite. Then for any $s$, on the images of the operator $(s\cdot)$ acting on $\obot\pi_a\KR(L,\phi)$, its inverse has an operator norm bounded from above by $\max_{a\in\Sh}\{\abs{s}_a^{-1}\}$. The same holds for $(s\cdot)$ acting on $\KR(L,\phi)$ and on $\obot_{a\in\Sh}\KC_a$.
\end{proposition}
\begin{proof}
    For any element $f$ in $\obot\pi_a\KR(L,\phi)$, since 
    \[\algnorm{s\cdot f}_{\phi}=\algnorm{s\cdot f}_{\max\abs{\ndot}_a}=\max_{a\in\Sh}\{\abs{s\cdot f}_a\}=\max_{a\in\Sh}\{\abs{s}_a\abs{\pi_a f}_a\},\]
    the operator norm of the inverse operator $(s^{-1}\cdot)$ on the image is bounded from above by $\max_{a\in\Sh}\{\abs{s}_{a}^{-1}\}$, which finite as $\abs{s}_a\neq 0$ for all $a$ by the assumption. The same holds for $(s\cdot)$ acting on the other two Banach spaces as they are isometric subspaces.
\end{proof}

\begin{construction}
    On the reduction variety $\proj([\KR(L,\phi)]^{\sim})$, denoted by $\widetilde{X}$, the \emph{reduction line bundle} $\widetilde{L}$ is the tautological ample invertible sheaf from this reduction graded algebra, the set of irreducible components are denoted by $\{\widetilde{X}_a\}_{a\in\Sh}$, the restriction of $\widetilde{L}$ to these components are $\widetilde{L}_a$. Note that these reduction objects depend on the metric $\phi$. 
    
    Consider the section algebras $R(\widetilde{X}, \widetilde{L})$ and $R(\widetilde{X}_a, \widetilde{L}_a)$, along with the natural restriction map $r_a$. There are natural morphisms $\rho: [\KR(L,\phi)]^{\sim}\rightarrow R(\widetilde{X}, \widetilde{L})$ and $\rho_a: [\pi_a\KR(L,\phi)]^{\sim}\rightarrow R(\widetilde{X}_a, \widetilde{L}_a)$, defined by sending $\widetilde{s}$ to the corresponding section. For any $a\in\Sh$, $\rho_a$ is injective, as $\widetilde{s}(\widetilde{a})\neq 0$ for $\widetilde{s}\neq 0$ by the description of a non-zero element in $[\pi_a\KR(L,\phi)]^{\sim}$. 
    Denote $\dim_{\widetilde{k}}[\pi_a\KR(L,\phi)]^{\sim}$ by $\chi_{a}(n)$ and $\dim_{\widetilde{k}}[\KR(L,\phi)]^{\sim}$ by $\chi(n)$.
\end{construction}

\begin{proposition}\label{P: localization section algebra HS}
    Assume that $\Shbd \KR(L,\phi)$ is a finite set. Then the homomorphsim of graded $\widetilde{k}$-algebras $[\iota_-]^{\sim}: [\KR(L,\phi)]^{\sim}\rightarrow \prod_{a\in\Sh}[\pi_a\KR(L,\phi)]^{\sim}$ have codimension asymptotic to $C\cdot n^d$ at level $n$ as $n\to\infty$. The multiplicities defined by the following limits exist and satisfy \[\lambda_a:=\lim_{n\to\infty}\chi_a(n)/\chi(n),\quad \sum_a\lambda_a=1.\]
\end{proposition}
\begin{proof}
    Using the commutative diagram given by $(\prod_a r_a)\circ\rho=(\prod_a \rho_a)\circ[\iota_-]^{\sim}$ and the injectivity of $\prod_a\rho_a$ and $[\iota_-]^{\sim}$, one bounds from the above the codimension of $[\iota_-]^{\sim}$ by the sum of codimensions of $\rho$ and $\prod_a r_a$. 
    
    As $\widetilde{L}$ is ample, and the graded $k$-algebra $R(L)$ is of Krull dimension $(d+1)$, the graded $\widetilde{k}$-algebras $[\KR(L,\phi)]^{\sim}$ and $R(\widetilde{X}_a, \widetilde{L}_a)$ are of Krull dimensions at most $(d+1)$, and section algebra on intersections $R(\widetilde{X}_a\cap \widetilde{X}_b, \widetilde{L}_{ab})$ is of Krull dimension at most $d$. Thus $\prod_a r_a$ has codimension $O(n^d)$. As $[\KR(L,\phi)]^{\sim}$ is reduced with its HS function being a polynomial, $\rho$ has codimension at most $O(n^d)$. Finally the codimension of $[\iota_-]^{\sim}$ is of order $O(n^d)$.

    The multiplicity $\lambda_a$ exists, and coincides with the leading term of the geometric HS polynomial for $R(\widetilde{X}_a,\widetilde{L}_a)$.
\end{proof}
\begin{remark}\label{R: codim bound}
    Note that the constant $C$ in the codimension bound only depends on $(L,\phi)$, though do not have an explicit expression.
\end{remark}

\subsection{Equidistribution measure and determinant norm distorsion}~
\subsubsection{Shilov finite case}
\begin{definition}\label{D: equidistribution measure Shilov finite}
    Let $\phi$ be a metric on on $L^{\an}$ which is Shilov finite. The probability measure $\sum_{a\in \Shbd \KR(L,\phi)}\lambda_a\delta_a $ on $\KM(\KR(L,\phi))$ is called the \emph{equidsitribution measure} of $\phi$ and is denoted by $\mu_{\eq}(\phi)$. One also use this name for the measure $p^{\an}_*\mu_{\eq}(\phi)$ on $X^{\an}$.
\end{definition}
\begin{remark}
    If $\phi$ is a diagonalizable Fubini-Study metric (i.e. an ample model metric as $\abs{k}=\bbR$), then $\mu_{\eq}(\phi)=\MA(\phi)$ by \cite[Théorème 6.9.3]{CLD}.
\end{remark}

\begin{theorem}\label{T: equidistribution limit}
    Let $\phi$ be a Shilov finite metric on $L^{\an}$. Let $s\in R_1(L)$ be a section, then the following limit formula holds
    \[\lim_{n\to\infty}\frac{1}{\chi(n)}\log~ \norm{\bigwedge^{\chi(n)}(\cdot s)}_{\hom,\phi}=\int_{\KM(\KR(L,\phi))}\log\abs{s}\mu(\phi)=\int_{X^{\an}}\log\abs{s}_{\phi}\ p^{\an}_*\mu(\phi).\]
\end{theorem}
\begin{proof}
    It suffices to prove the first equality. The basic idea is to compute the operator norm on a slightly bigger space $\obot_a \KC_a$, where the operator $(\cdot s)$ can be analytically localized and decoupled. Thanks to Proposition \ref{P: det norm distorsion subspace} and Proposition \ref{P: localization section algebra HS}, the error will be of order $C\cdot n^{d}$ where $C$ is a constant independent of $n$, hence has no effect in the limit.

    For each Shilov point $a$, consider the normed $k$-vector subspace $\KC_a$ in $(\widehat{\kappa}(a),\abs{\ndot}_a)$ and the linear map $(\cdot s)$. By Proposition \ref{P: VF extension operator norm}, the operator norm of $\bigwedge(\cdot s)|_{\KC_a}$ is equal to $\abs{s}_a^{\chi_{a}(n)}$, whose $\log$ gives the summands on the right hand side.
    

\end{proof}
\begin{corollary}
    Same setting as above. Denote by $\mathrm{err}(n)$ the error term in the above limit equality. Then as $n\to\infty$, it holds \[\inf_{a\in\Shbd}\{\log\abs{s}_a\}\frac{C}{n}\leq\mathrm{err}(n)\leq\sup_{a\in\Shbd}\{\log\abs{s}_a\}\frac{C}{n} \]
    where $C$ is a constant only depending on $(L,\phi)$, not on $s$.
\end{corollary}
\begin{proof}
    This is a consequence of Proposition \ref{P: det norm distorsion subspace} and Remark \ref{R: codim bound}.
\end{proof}

Recall that for two graded algebra norms $\algnorm{\ndot}$ and $\algnorm{\ndot}'$  on $R$, one can define the following two pseudo-distances, which satisfies $d_1\leq d_{\infty}$
\[d_{\infty}(\algnorm{\ndot},\algnorm{\ndot}'):=\lim\frac{1}{n}\dist(\norm{\ndot}_{n}, \norm{\ndot}'_{n}),d_{1}(\algnorm{\ndot},\algnorm{\ndot}'):=\lim\frac{1}{\chi(n)}\dist(\bigwedge^{\chi(n)}\norm{\ndot}_n, \bigwedge^{\chi(n)}\norm{\ndot}'_n).\]

\begin{construction}\label{Cs: d_1-topology}
    Let $\Prob_{\fin,\bir}$ be the set of probability measures on $\Cone(L)^{\an}$ whose support is a finite set of birational points. The following formula
    \[d(\mu,\nu):=\sup_{n\in\bbN, s\in R_n(L)\setminus\{0\}}\abs{\int_{\Cone(L)^{\an}}\frac{1}{n}\log\abs{s}(\mu-\nu)}\]
    defines a pseudo-distance satisfying the (Archimedean) triangle inequality.  
\end{construction}

\begin{proposition}\label{P: continuity of equidistribution measure}
    Let $\phi_{1,2}$ be two Shilov finite metrics on $L^{\an}$, then it holds
    \[d(\mu_{\eq}(\phi_1),\mu_{\eq}(\phi_1))\leq 2 d_1(\algnorm{\ndot}_{\phi_1},\algnorm{\ndot}_{\phi_2}).\]
    The map $\algnorm{\ndot}\mapsto \lim_{n\to\infty}\frac{1}{\chi(n)}\log~ \norm{\bigwedge^{\chi(n)}(s\cdot)|_{R_n}}_{\hom, \FS_{\algnorm{\ndot}}}$ is continuous (for both the topologies induced by $d_1$ or $d_{\infty}$).
\end{proposition}
\begin{proof}
    Write $\phi_i=\FS(\algnorm{\ndot}_i)$ for $i=1,2$, by Proposition \ref{P: operator norm changes}, it holds that  
    \[ \big|\log~ \norm{\bigwedge^{\chi(n)}(\cdot s)}_{\hom,\phi_1}-\log~ \norm{\bigwedge^{\chi(n)}(\cdot s)}_{\hom,\phi_2}\big|\leq A(n)+B(n),\forall n\in\bbN, \] 
    \[A(n)=\dist(\bigwedge^{\chi(n)}\norm{\ndot}_{(n+1)\phi_1}|_{sR_n}, \bigwedge^{\chi(n)}\norm{\ndot}_{(n+1)\phi_1}|_{sR_n}),\ B_n=\dist(\bigwedge^{\chi(n)}\norm{\ndot}_{n\phi_1}, \bigwedge^{\chi(n)}\norm{\ndot}_{n\phi_1}),\]
    by taking sup for $s$ and limit for $n$, one gets the desired inequality. The second assertion also follows from the above estimate.
\end{proof}

\subsubsection{General case}
\begin{construction}\label{Cs: function cone}
    Consider the topological space $\Cone(L)^{\an}$ and functions on it. Let $\overline{\underline{\bbR}}$ be the extended real number $\bbR\cup\{\pm\infty\}$. Denote by $V$ the $\bbR$-vector space generated by functions (valued in $\overline{\underline{\bbR}}$) of the form $\log\abs{s}$ for all homogeneous $s\in R_n\setminus\{0\}$. For some subset $\Delta$, denote by $\theta_{\Delta}(\ndot)$ the sup norm ($\overline{\underline{\bbR}}$-valued) on $\Delta$. 
    
    If all points of $\Delta$ are birational, then $\theta_{\Delta}$ is a (usual $\bbR$-valued) norm on $V$; its completion, denoted by $\KV_{\Delta}$, is a Banach space, which admits an isometric injective linear map to the Banach space $\KC(\Delta)$: associate to a Cauchy sequence $\{v_i\}_{i\in\bbN}$ in $(V,\theta_{\Delta})$ the element $\lim_{n\to\infty}v_i|_{\Delta}$. On $V$, this map is the restriction to $\Delta$.
\end{construction}

\begin{construction}\label{Cs: equidistribution measure linear form}
    Let $\phi$ be a semipositive metric. For any homogeneous $s\in R$, and any sequence of Shilov finite metrics $\{\phi_i\}_{i\in\bbN}$ approximating $\phi$ in the topology of norms, by Proposition \ref{P: continuity of equidistribution measure} the sequence $\{L_{\phi_i}(\log\abs{s})\}_{i\in\bbN}$ converges to $L_{\phi}(\log\abs{s})$, where the functional is denoted by
    \[L_{\phi}(\log\abs{s}):=\lim_{n\to\infty}\frac{1}{\chi(n)}\log \norm{\bigwedge^{\chi(n)}(\cdot s)}_{\hom,\phi}.\]
    It is a linear functional in the function $\log\abs{s}$ on $\Cone(L)^{\an}$, hence can be readily extended to a linear form on $V$. Moreover, it is bounded from above by $\theta_{\Delta}(\ndot)$ if $\Delta$ is a Shilov boundary of $\KR$. Call it the \emph{equidistribution linear functional} (of the metric $\phi$, or of the graded algebra norm $\algnorm{\ndot}_{\phi}$).
    
    Thus if there is a Shilov boundary $\Shbd\KR$ whose points are all birational, then one can extend $L_{\phi}$ to the Banach subspace $\KV_{\Shbd\KR}$ of $\KC(\Shbd\KR)$. By Hahn-Banach extension theorem and Riesz representation theorem, the linear form defines a measure on $\Delta$, which is positive as $L_{\phi}$ is. Denote by $\mu_{\eq,\Shbd\KR}(\phi)$ this measure (however it is not necessarily unique). Call it an \emph{equidistribution measure} of $\phi$ (depending on the Shilov boundary $\Shbd\KR$, and also on the choice of H-B extension). 
\end{construction}
\begin{proposition}\label{P: equidistribution measure general semipositive}
    Let $\phi$ be a semipositive metric. Suppose there is a Shilov boundary $\Delta$ whose points are all birational, then there exists a probability measure $\mu_{\eq,\Delta}(\phi)$ supported on $\Delta$, such that 
    \[\lim_{n\to\infty}\frac{1}{\chi(n)}\log~ \norm{\bigwedge^{\chi(n)}(\cdot s)}_{\hom,\phi}=\int_{\KM(\KR(L,\phi))}\log\abs{s}\mu_{\eq,\Delta}(\phi).\]
\end{proposition}
\begin{remark}
    At this stage, it is not clear how this equidistribution measure actually depends on the choice of a Shilov boundary, nor how to assure in general the assumption on the birationality of points in it, apart from cases where $\Delta$ is contained in some skeleton. Note that even if one chooses to use the canonical Shilov boundary (Appendix A), it is not clear whether $\KV_{\Delta}$ is equal to $\KC(\Delta)$, hence not clear whether the corresponding equidistribution measure is unique. 
\end{remark}


\subsection{Monge-Ampère measure}

Now we pass to continuous semipositive metrics. Recall the definition and the continuity property of the Monge-Ampère operator.

\begin{construction}\label{Cs: MA measure}
    To any continuous semipositive metric $\phi$, one can associate a probability Radon measure $\MA(\phi)$ on $X^{\an}$ written as $(\ddc\phi)^{\wedge d}$; this map is continuous with respect to the uniform topology on the space of continuous semipositive metrics and the weak topology of measures: if $\phi_n\rightarrow\phi$ uniformly, then $\MA(\phi_n)\rightarrow\MA(\phi)$ weakly \cite{CL}\cite{CLD}\cite{GK}. 
\end{construction}
\begin{construction}\label{Cs: lifted MA measure}
    Let $\phi$ be a continuous semipositive metric, denote by $\MA^{\uparrow}(\phi)$ the measure $(\ddc \max\{\phi,0\})^{\wedge(d+1)}$ on $\Cone(L)^{\an}$. Denote by $\gamma_\phi: X^{\an}\to \partial\KM(\KR(\phi))$ the continuous map $x\mapsto \gamma_{\phi}(x)$, which is a homeomorphism since $\phi$ is continuous. 
\end{construction}
\begin{proposition}\label{P: lifted MA measure}
    Let $\phi$ be a continuous semipositive metric. The measure $\MA^{\uparrow}(\phi)$ is supported on $\gamma_{\phi}(X^{\an})$, and its direct image by $\gamma_\phi$ is the measure $\MA(\phi)$ on $X^{\an}$. If $\phi_n\to\phi$ converges uniformly, then $\MA^{\uparrow}(\phi_n)\to \MA^{\uparrow}(\phi)$ converges weakly as probability measures on $\Cone(L)^{\an}$.
\end{proposition}
\begin{proof}
    Approximate $\phi$ uniformly by Fubini-Study metrics $\phi_i$ for which the properties holds, and pass to the limit.
\end{proof}

\begin{theorem}\label{T: eq=MA for Shilov finite}
    Let $\phi$ be a continuous semipositive metric, which is Shilov finite. Then $\mu_{\eq}(\phi)=\MA^{\uparrow}(\phi)$.
\end{theorem}

\begin{lemma}
    These two measures have same support.
\end{lemma}
\begin{proof}
    The equality for supports holds for Fubini-Study metrics. Let $\{\phi_n\}$ be a sequence of Fubini-Study metrics approximating $\phi$ uniformly. Consider any $s\in R_n\cap\KR^{\circ}$, on the one hand, thanks to Proposition \ref{P: continuity of equidistribution measure}, one gets (write $\int$ for $\int_{\Cone(L)^{\an}}$ for all below)
    \[\int\log\abs{s}\mu_{\eq}(\phi_n)\xrightarrow{n\to\infty} \int\log\abs{s}\mu_{\eq}(\phi);\]
    on the other hand, for any $\delta>0$, denote by $\log_{\delta}\abs{s}$ the continuous function $\max\{\log\abs{s},\log\delta\}$, by the continuity of $\MA^{\uparrow}(\ndot)$, it holds that
    \[\int\log_{\delta}\abs{s}\MA^{\uparrow}(\phi_n)\xrightarrow{n\to\infty}\int\log_{\delta}\abs{s}\MA^{\uparrow}(\phi).\]
    Since $\phi$ is Shilov finite, for any neighbourhood $U$ of $\Shbd\KM(\KR(\phi))$, there exists a homogeneous peak section $p\in\KR^{\circ}$ such that $\abs{p(\xi)}=1$ for all $\xi\in\Shbd$ and the locus $P(\abs{p})$ is contained in $U$. Thus taking $\limsup_{n\to\infty}$ for the following terms
    \[\int\log\abs{p}\mu_{\eq}(\phi_n)=\int\log\abs{p}\MA^{\uparrow}(\phi_n)\leq \int\log_{\delta}\abs{p}\MA^{\uparrow}(\phi_n)\]
    one gets 
    \[0=\int\log\abs{p}\mu_{\eq}(\phi)\leq \limsup \int\log_{\delta}\abs{p}\MA^{\uparrow}(\phi)\leq 0,\]
    thus (by also considering $\liminf$) one gets $\int\log_{\delta}\abs{p}\MA^{\uparrow}(\phi)=0$ for any $\delta\in(0,1)$, hence the support of $\MA^{\uparrow}(\phi)$ is contained in $\abs{p(\ndot)}=1$ which is in $U$. As $U$ can be made arbitrarily small, one knows that $\supp(\MA^{\uparrow}(\phi))\subseteq\Shbd\KM(\KR(L))$. 
    
    This inclusion is in fact an equality: otherwise, by Proposition \ref{P: finite Shilov boundary basic} one can find a homogeneous section $p$ such that $\abs{p(\xi)}=1$ for all $\xi\in\supp(\MA^{\uparrow}(\phi))$ and $\abs{p(\xi')}<1$ for other $\xi'\in \Shbd$; test the two measures against $\log\abs{p}$ gives 
    \[0=\int \log\abs{p}\MA^{\uparrow}(\phi)=\int \log\abs{p}\mu_{\eq}(\phi) <0\]
    a contradictory inequality.
\end{proof}
\begin{proof}
    Since both measures can be written as convex combinations of Dirac measures supported on a common finite set $\Shbd\KR$, it suffices to equalize the coefficients. For any $\xi\in \Shbd\KR$, by Proposition \ref{P: finite Shilov boundary basic} one can find $p\in R(L)$ such that $\abs{p(\xi)}<1$ and $\abs{p(\xi')}=1$ for all other $\xi'\in \Shbd\KR$. Test these two measures against the function $\log\abs{p}$ gives equality of coefficients for $\Delta(\xi)$.
\end{proof}

\begin{corollary}\label{T: equidistribution semipositive}
    Let $\phi$ be a continuous semipositive metric. Then
    \[\lim_{n\to\infty}\frac{1}{\chi(n)}\log~ \norm{\bigwedge^{\chi(n)}(\cdot s)}_{\hom,\phi} =\int_{X^{\an}}\log\abs{s}_{\phi}\MA(\phi).\]
\end{corollary}
\begin{remark}
    One would also like to compare $\mu_{\eq}(\phi;\Delta)$ with $\MA^{\uparrow}(\phi)$ when $\phi$ is continuous and $\Delta$ consists of birational points. We shall address this problem in a future study using tropical geometry. At the current stage, it is only known that they define the same linear form on the function space $V$.
\end{remark}

\section{Examples and applications}

\subsection{Locally affinoid Banach algebra and semipositive piecewise linear metrics}
There are two natural classes of semipositive metrics whose Shilov boundary consists of finitely many birational points. Firstly (formal) model (semipositive) metrics have this property, and there's a characterization for such metrics in terms of the Banach algebra $\KR(L,\phi)$ being locally affinoid, based on a result in \cite{GK}. Secondly, the envelope construction yields (possibly singular) metrics with almost prescribed Shilov boundaries (to be discussed in the next subsection).

Recall that the spectrum of a Banach algebra is called locally affinoid, if it is a $G$-domain, namely a finite union of affinoid domains \cite{Tem}.
\begin{proposition}\label{P: Stein-Liu metric}
    If $\phi$ is a metric whose disc domain $\KM(\KR(L,\phi))$ is locally affinoid, then $\phi$ is Shilov finite.
\end{proposition}
\begin{proof}
    By definition, this domain admits a finite cover by affinoid domains $\{V_i\}_{i\in I}$, so the Shilov points for $\KM(\KR(L,\phi))$ is contained $\bigcup_{i\in I}\Shbd V_i$ which is a finite set. 
\end{proof}

Recall that a formal model $(\fX,\fL)$ over $k^{\circ}$ for $(X,L)$ defines a model metric $\phi_{\fL}$ \cite{Gub}: there is a covering by formal open subsets $\{\fU_i\}$ of $\fX$ such that sections of $\fL|_{\fU_i}$ are of constant norm one (or equivalently a $G$-covering by affinoid domains $U_i$ of $X^{\an}$ with local sections $s_i$ of $L^{\an}|_{U_i}$ of constant norm one). An equivalent characterization given in \cite[Proposition 8.11]{GK} for such metrics is that $\phi$ is piecewise linear, in the sense that $U_i$ can be $G$-covered by affinoid domains $\{W_{ij}\}$ on which there exist tropical moment maps $F_{ij}$ to algebraic torus $T_{ij}^{\an}$ that pulls a linear function $\phi_{ij}$ to the function $\log\abs{s_i}_{\phi}$ restricted on $W_{ij}$. (Note that we assumed $\abs{k^*}=\mathbb{R}_+$ so we do not take care of rationality of slopes in the definition of piecewise linear objects.)


\begin{proposition}\label{P: piecewise-linear is equivalent to Stein-Liu}
    The followings are equivalent for a metric $\phi$ on the ample line bundle $L$: (1) it is defined by a vertically nef formal model of $L$; (2) it is piecewise linear and semipositive; (3) $\phi$ is continuous, the holomorphically convex set $\KM(\KR(L,\phi))$ is locally affinoid, and is equal to $\KM^-(\KR(L,\phi))$.
\end{proposition}
\begin{proof}
    That $(1)\Leftrightarrow (2)$ is part of \cite[Proposition 8.11]{GK} and \cite[Theorem 6.10]{GK2}.
    
    For $(1)\Rightarrow (3)$, first suppose that $\phi$ is given by a vertically nef formal line bundle of $L$, then there is a cover by affinoid domains $\{U_i\}$ of $X^{\an}$ and constant norm one local sections $\{s_i\}$ of $L$. The dual unit disc bundle of $(L^{\otimes n})^{\an}$ over $U_i$ of $n\phi$ is isomorphic to $U_i\times \overline{\bbD}^1$ via the isomorphism given by $s_i$, hence is an affinoid domain. Its image under the affine map $p^{\an}$ is thus a locally affinoid domain in $\Cone(L)^{\an}$. Thus $\KM(\KR(L,\phi))$ is covered by a finite number of affinoid domains, hence is locally affinoid. More generally, if $n\phi$ is given by a vertically nef formal line bundle, then $\KM(n\phi)$ in $\Cone(L^{\otimes n})^{\an}$ is locally affinoid; as the tautological map induced by $R^{[n]}\to R$ is finite and induces the map $\KM(\phi)\to \KM(n\phi)$, the domain $\KM(\phi)$ is also locally affinoid as this property is preserved by finite morphisms.

    For $(3)\Rightarrow(2)$, take a $G$-cover of $\KM(\KR(L,\phi))$ by finitely many irreducible affinoid domains $\{V_i\}$. As $\abs{s}_{\phi}(x)$ can be written as $\max_{V_i\ni x}\max_{z\in V_i\cap (L^{\vee})^{\an}(x)}\{\abs{s(z)}\}$, the metric is thus piecewise linear.
\end{proof}

\begin{remark}
    Locally affinoid but non-affinoid uniform Banach algebras can be obtained via the affinoid pinching construction \cite{Tem}. Roughly speaking, the disc domain $\KM(\KR(L,\phi))$ is affinoid if and only if $\phi$ is induced by an ample model of $L$; 
\end{remark}

\subsection{Shilov finite metric and envelope}
Another way to obtain Shilov finite metric is via the envelope construction. Note that this produces a possibly \emph{singular} metric, in the sense that it could be discontinuous or degenerate at some point.

\begin{construction}\label{Cs: envelope finite points}
    Let $\{z_i\}_{i\in I}$ be a finite set of birational points in $\Cone(L)^{\an}$ which are graded (i.e. $z_i^+=z_i$, see Construction \ref{Cs: gradization}), with $\abs{I}=m$. Denote by $I_{\Sh}\subset I$ the subset of Shilov points of the domain $\KP(\{z_i\}_{i\in I})$. Let $\Delta_I$ be the standard simplex $\{ \underline{x}\in\bbR_+^{\abs{I}},\sum x_i=1\}$, and $\Delta_J$ be its face for any $J\subset I$ with the inclusion map $\iota_{J,I}$. Denote by $\underline{0}$ the origin $(0,\dots,0)$ of $\bbR_+^{\abs{I}}$.
\end{construction}
\begin{construction}
    Fix a homogeneous element $s\in R(L)_1$, this gives a coordinization $(\abs{s(z_1)},\dots,\abs{s(z_m)})\in\bbR_+^m$ denoted by $\underline{t}$ (we omit notating $s$). Let $\algnorm{\ndot}(\underline{t})$ be the algebra norm $\max_{i\in I}\abs{\ndot}_{z_i(t_i)}$ on $R(L)$ which is also graded. Denote by $\phi(\underline{t})$ the  metric $\KP(\algnorm{\ndot}(\underline{t}))$. 
    
    Denote by $\underline{\lambda}(\underline{t})$ the $m$-tuple $(\lambda_1,\dots,\lambda_m)\in\Delta_I$ the coefficients in the equidistribution measure $\mu(\phi(\underline{t}))$. Note that for any $r\in\bbR_+$, one has $\algnorm{\ndot}(r\underline{t})=r.\algnorm{\ndot}(\underline{t})$ and $\phi(r\underline{t})=r\phi(\underline{t})$ hence $\mu(r\underline{t})=\mu(\underline{t})$. We define a map from norms to measures
    \[\pmb{\mu}: \Delta_I\simeq \big(\bbR_+^{\abs{I}}\setminus\{\underline{0}\}\big)/\bbR_+^{\times}\rightarrow \Delta_I, \quad\underline{t}\mapsto \underline{\lambda}(\underline{t}). \]
    For each fixed $\underline{t}$, it factorizes through the face inclusion map $\iota_{I_{\Sh}(\underline{t}),I}: \Delta_{I_{\Sh}(\underline{t})}\to \Delta_{I}$
    as the equidistribution measure is supported Shilov points.
\end{construction}

\begin{proposition}\label{P: continuity coefficient of measure}
    The map $\pmb{\mu}$ is continuous. 
\end{proposition}
\begin{proof}
    By Theorem \ref{T: equidistribution limit}, for any homogeneous $f$, the following equation holds 
    \[L_{\phi(\underline{t})}(\log\abs{f}):= \frac{1}{\chi(n)}\log~ \norm{\bigwedge^{\chi(n)}(\cdot s)}_{\hom,\phi(\underline{t})}= \sum\lambda_i(\underline{t})\log\abs{f(z_i(t_i))}.\]
    We shall put in several $f$ as test functions, to form a linear system and to solve it. By Proposition \ref{P: operator norm changes}, the second member (LHS) is continuous in $\underline{t}$, it suffices to choose carefully $\{f_i\}_{i\in I}$ such that the matrix $M(\underline{t})$ formed by $\{\log\abs{f_j(z_i(t_i))}\}$ is (invertible and) continuous in $\underline{t}$. Note that these choices can depend on $\underline{t}$ and it suffices to have local continuity near any given $\underline{t}$.
    \begin{lemma}\label{L: upper triangular peak choices}
        For any $\underline{t}\in\bbR_+^{\abs{I}}\setminus\{\underline{0}\}$, and any $i\in I$, denote by $\frak{p}_i$ the prime ideal corresponding to $z_i(t_i)$ via the norm reduction construction. One can rearrange $I$ in an order such that if $j<i$, then $\frak{p}_i\nsubseteq \frak{p}_j$; then there exists $(f_i)_{i\in I}$ such that for any $i\in I$, $\abs{f_i(z_j(t_j))}=1$ for all $ j<i$ and $\abs{f_i(z_i(t_i))}<1$.
    \end{lemma}
    \begin{proof}
        Take $I_1$ such that the corresponding ideals are the maximal elements for the inclusion partial order, then take $I_2$ whose ideals are the maximal ones among the rest, and so on to get $I_n$; arrange $I$ in the order $(I_n, I_{n-1},\dots, I_1)$, the elements inside each $I_m$ can be arranged in any order; this achieves the goal.

        By the construction, the prime $\frak{p}_i$ is not contained in any $\frak{p}_j$ for $j<i$, hence it is not contained in the union $\bigcup_{j<i}\frak{p}_j$ by prime avoidance, so one can find an element $\widetilde{f_i}\in\widetilde{\KR}$ in $\frak{p}_i\setminus \bigcup_{j<i}\frak{p}_j$. Any lift $f_i\in\KR^{\circ}$ satisfies the requirement.
    \end{proof}
    It follows from the lemma that the matrix $M(\underline{t}):=\{\log\abs{f_i(z_j(t_j))}\}_{i,j\in I^2}$ is upper triangular with non-zero elements on the diagonal; hence for all $\underline{t}'$ sufficiently near $\underline{t}$, the matrix $M(\underline{t}')=\{\log\abs{f_i(z_j(t'_j))}\}_{i,j\in I^2}$ is still invertible, and is continuous in $\underline{t}'$. Since one has the linear system $\underline{\lambda}(\underline{t}')\cdot M(\underline{t}')=L_{\phi(\underline{t}')}(\log\abs{f_1},\dots, \log\abs{f_n})$ and the right hand side is continuous in $\underline{t}'$ near $\underline{t}$, so is the solution $\underline{\lambda}(\underline{t}')$.
    
\end{proof}

\begin{proposition}\label{P: surjectivity of equidistribution measure coefficients}
     The map $\pmb{\mu}$ is surjective. Thus for any $\underline{\lambda}\in\Delta_I$, there exists $\underline{t}\in\Delta_I$ that solves the equation $\mu_{\eq}(\phi(\underline{t}))=\sum\lambda_i\delta_{z_i}$. In addition, if all $\abs{\ndot}_{z_i}$ are bounded from below as graded algebra norms, then the solutions are non-degenerate (but possibly singular) metrics. 
\end{proposition}
\begin{proof}
    One argues by induction on the number of points in Shilov boundary, using the above continuity property and the mapping degree theorem. For each $i\in I$ and $t_i\in\bbR_+$, as the domain $\KP(\{z_i(t_i)\})$ has Shilov boundary $\{z_i(t_i)\}$, the equidistribution measure of the metric associated to $\algnorm{\ndot}(\iota_{\{i\},I}(t_i))$ is supported on $\{z_i(t_i)\}$. Thus $\pmb{\mu}$ restricts to a map $\Delta_{\{i\}}\to \Delta_{\{i\}}$.

    Similarly, for any fixed pair $J=\{i,j\}\subset I$, and any $(t_i, t_j)\in\bbR_+^2\setminus (0,0)$, the domain $\KP(\{z_i(t_i), z_j(t_j)\})$ has a Shilov boundary contained in $\{z_i(t_i), z_j(t_j)\}$, the equidistribution measure of the metric associated to $\algnorm{\ndot}(\iota_{\{i,j\},I}(t_i, t_j))$ is supported on $\{z_i(t_i), z_j(t_j)\}$. Thus map $\pmb{\mu}$ is restricted to a self-map $\Delta_{J}\to \Delta_{J}$, which is continuous and is surjective when further restricted on the boundary $e_i$ and $e_j$. Hence the restriction to $\Delta_J$ is surjective by the mapping degree theorem. 

    Using induction, the same arguement shows that, for any $J\subset I$, by construction $\pmb{\mu}$ restricts to a self map on $\Delta^m_{J}$, which is surjective when further restricted to its boundaries $\partial \Delta^m_{J}$ by the induction hypothesis. Therefore by the mapping degree theorem, the restriction of $\pmb{\mu}$ to $\Delta^m_{J}$ is also surjective. 

    The last assertion is immediate since for any $\underline{t}$, one can choose $i\in I$ such that $t_i\neq 0$, and the graded algebra norm $\algnorm{\ndot}(\underline{t})$ is bounded from below by $\algnorm{\ndot}(\iota_{\{i\}, I}(t_i)))$.
    
    
    
\end{proof}
\begin{remark}
    Note that we've solved the equation with prescribed (Shilov finite) equidistribution measure without using the regularization/continuity property of envelope metrics \cite{BFJ2}\cite{GJKM}\cite{FGK}; one needs that crucial regularity property in order to compare this equation to a genuine Monge-Ampère equation. 
\end{remark}

\appendix
\section{Minimal and Shilov boundary} 

Here we give a recount of the main result and some related constructions in \cite{Gue}\cite{EM}, where an `almost minimal' boundary $\Min(\KB)$ is constructed; together with a property of peak isolated points stating that they are contained in this boundary, a statement slightly finer than Proposition \ref{P: peak point if red saturated}(2). 
\begin{definition}\label{Cs: S-localization of norms}
    Let $\KB$ be a uniform Banach algebra, with the spectral/sup norm $\norm{\ndot}_{\Sp}$ denoted by $\theta$. Denote by $\KF$ the set of $\bbR_+$-valued functions on $\KB$, view the spectral norm as an element $\theta\in\KF$. For $\alpha\in \KF$, say it is \emph{sub-multiplicative} (resp. \emph{power-multiplicative}, resp. \emph{multiplicative}) according to the usual sense. Let $S$ be a sub-semigroup of $(\KB,\times)$, and $\alpha\in\KF$ a sub-multiplicative function. Define its \emph{S-smoothing} $\alpha^S$ as $g\mapsto\inf_{f\in S}\alpha(fg)/\alpha(f)$. Say $\alpha$ is \emph{S-multiplicative} if $\alpha(fg)=\alpha(f)\alpha(g)$ for any $f,g\in S$; in this case $\alpha^S(g)=g$ for any $g\in S$.
\end{definition}
\begin{construction}
    Let $\Delta\in\KM(\KB)$ be a closed subset. Denote by $S(\Delta)$ the subset of $\KB$ consisting of elements $g\in\KB$ such that $\abs{g(z)}=\norm{g}_{\Sp}$ for any $z\in \Delta$. Then $S(\Delta)$ is a sub-semigroup of $(\KB,\times)$. Denote by $\theta_{\Delta}\in\KF$ the semi-norm $\sup_{z\in\Delta}\abs{\ndot}_z$.
    In particular, if $f\in\KB^{\circ\setminus\circ\circ}$, then $P(\abs{f})$ is a non-empty closed set in $\KM(\KB)$. Then $S(P(\abs{f}))\cap\KB^{\circ\setminus\circ\circ}$ are those elements $h$ with $\norm{h}_{\Sp}=1$ and $D(\widetilde{f})\subset D(\widetilde{h})$.
\end{construction}
\begin{proposition}\label{P: analytic localization by S}
    Let $f$ be an element in $\KB^{\circ\setminus\circ\circ}$ and $\Delta\in\KM(\KB)$ be a closed subset. Suppose that $P(\abs{f})\subset \Delta$. The $(\theta_{\Delta})^{S(P(\abs{f}))}=\theta(P(\abs{f}))$.
\end{proposition}
\begin{proof}
    For any $g\in\KB$, one needs to show that 
    \[\inf_{h\in S(P(\abs{f}))}\big(\sup_{w\in\Delta}\abs{hg(w)}/\sup_{z\in\Delta}\abs{h(z)} \big)=\sup_{w\in P(\abs{f})}\abs{g(w)}.\]
    This follows from inequalities by construction
    \begin{equation*}
        \begin{split}
            &\big(\sup_{w\in\Delta}\abs{hg(w)}/\sup_{z\in\Delta}\abs{h(z)} \big)=\big(\sup_{w\in\Delta}\abs{h(w)g(w)}/\sup_{z\in P(\abs{f})}\abs{h(z)} \big)\\ 
            \geq &\sup_{w\in P(\abs{f})}\abs{h(w)g(w)}/\sup_{z\in P(\abs{f})}\abs{h(z)}=\sup_{w\in P(\abs{f})}\abs{g(w)}
        \end{split}
    \end{equation*}
    hence LHS$\geq$RHS. On the other hand, note that $\abs{f(w)}<1$ for any $w\in\Delta\setminus P(\abs{f})$. Fixing $g$, for any neighbourhood $V$ of $P(\abs{f})$ in $\Delta$, one can find $n\in\bbN$ such that 
    \[\sup_{w\in P(\abs{f})}\abs{f^n(w)g(w)}\leq \sup_{w\in V}\abs{g(w)},\]
    hence the $\inf$ of LHS (tested by $f^n$) is smaller than $\sup_{w\in V}\abs{g(w)}$. Thus LHS$\leq$RHS by considering arbitrarily small neighbourhoods. 
\end{proof}
\begin{remark}
    The key fact is that $\abs{f}$ is a `peak function' with pics exactly on (every point of) the subset $P(\abs{f})$. Hence the $S(P(\abs{f}))$-smoothing process selects this subset in an `analytic localization' procedure. 
\end{remark}

\begin{definition}
    Equip $\KF$ with the partial order induced by the pointwise total order on $\bbR_+$: $\alpha\leq \beta$ if and only if $\alpha(f)\leq \beta(f)$ for all $f\in \KB$. A subset $N$ (with the induced partial order) \emph{is good} if every subset of $N$ has a maximum; in this case, for any $\alpha\in N$, write $\alpha_{-}$ the maximal element in $\{\gamma\in N, \gamma<\alpha\}$ if this subset is non-empty.
\end{definition}
\begin{example}
    Consider $\Delta$ and $\Gamma$ two closed subsets of $\KM(\KB)$. Then $\theta_{\Delta}\leq\theta_{\Gamma}$ if $\Delta\subseteq \Gamma$. Note that the reverse implication does not hold without further assumptions: $\theta_{\Shbd}=\theta_{\Shbd\cup A}$ for any closed subset $A$.
\end{example}

\begin{proposition}\label{P: comparing partial order for sup norms}
    Let $\Delta$ and $\Gamma$ be two closed subsets of $\KM(\KB)$, both contained in $\Shbd(\KB)$. Then $\theta_{\Delta}\leq\theta_{\Gamma}$ if and only if $\Delta\subseteq \Gamma$; the strict inequality corresponds to the strict inclusion.
\end{proposition}

\begin{proposition}\label{P: comparing partial order for sup norms 2}
    Let $f$ and $g$ be elements in $\KB^{\circ\setminus\circ\circ}$. Then $\theta_{P(\abs{f})}\leq\theta_{P(\abs{g})}$ if and only if $P(\abs{f})\subseteq P(\abs{g})$; the strict inequality corresponds to the strict inclusion.
\end{proposition}

\begin{definition}\label{Cs: constructible/localizable subset}
    Let $N\subset\KF$ be a subset;  it is \emph{constructible} if: 
    \begin{enumerate}
        \item the set $(N,\leq)$ is good;
        \item for any $\alpha$ admitting an $\alpha_-$, there exists a sub-semigroup $S$ of $(\KB,\times)$ such that $\alpha$ is $S$-multiplicative and $\alpha^S=\alpha_-$;
        \item for any $\beta\in N$ which can not be $\alpha_-$ for some $\alpha$, it holds that $\beta=\inf\{\gamma\in N,\beta<\gamma\}$.
    \end{enumerate}
    Given $\alpha$, $\beta$ two multiplicative elements, say $\beta$ is \emph{constructible from} $\alpha$ if there is a constructible subset $N$ admitting $\beta$ as its minimum and $\alpha$ as its maximum. 
\end{definition}

The main result in \cite{Gue}, revisited by \cite{EM}, is the construction of an `almost minimal' boundary.
\begin{theorem}\label{T: minimal boundary as constructible}\cite{Gue}\cite{EM}
    Denote by $\Min(\KB,\theta)$ the subset of elements $\sigma$ in $\KF_{\theta}$ that satisfy: (0) $\sigma\leq \theta$; (1) $\sigma$ is multiplicative; (2) $\sigma$ is constructible from $\theta$. It holds that $\Min(\KB,\theta)$ is a boundary of $\KB$, and is contained in any other closed boundary. 
\end{theorem}
\begin{corollary}\label{C: unique Shilov boundary}
    The Shilov boundary of $\KB$ is unique: it is the topological closure of the boundary $\Min(\KB,\theta)$, denoted by $\Shbd\KB$. 
\end{corollary}

\begin{proposition}
    Let $\KB$ be a uniform Banach algebra. Suppose that $\KM(\KB)$ is an Urysohn-Fréchet space. Then any $z\in\KM(\KB)$ that is peak isolated lies in $\Min(\KB,\theta)$. 
\end{proposition}
\begin{proof}
    Let $z$ be such a point, one shall construct a constructible set in $\KF$ relating $\abs{\ndot}_z$ and $\norm{\ndot}_{\Sp}$ and use Theorem \ref{T: minimal boundary as constructible} as follows. Thanks to the intersection property for $P(\abs{\ndot})$, by transfinite induction, one can find a totally ordered set $(I,\prec)$ and a partial-order preserving map $i\mapsto f_i$ to the (opposite of) partially ordered set $(\widetilde{\KB}\setminus\frak{p}, \ndot|\ndot)$, such that $\{P(\abs{f_i})\}_{i\in I}$ is decreasing for $\subseteq$, and $\{z\}=\bigcap_{i\in (I,\prec)}P(\abs{f_i})$. If the set $I$ can be made countable (made isomorphic to $\bbN$), then the following subset of $\KF$ will satisfy all conditions in Theorem \ref{T: minimal boundary as constructible}
    \[N:=\{\theta_{P(\abs{f_i})},i\in I\},\ \theta_{P(\abs{f_i})}:=\sup_{w\in P(\abs{f_i})}\abs{\ndot}_w\]
    To check: (1) holds by Proposition \ref{P: comparing partial order for sup norms 2} and the assumption that $I$ is countable; (2) holds as $(\theta_{P(\abs{f_i})})^{S(P(\abs{f_j}))}=\theta_{P(\abs{f_j})}$ for any $j\prec i$ by Proposition \ref{P: analytic localization by S}; (3) holds trivially as in $N$, only $\theta_{P(\abs{f_0})}$ has no subsequent element and it is the maximal one. Thus we conclude $z\in \Min(\KB,\theta)$. 
    
    To make $I$ countable, one uses the assumption on topology. If $z$ is an isolated point in $\KM(\KB)$, by Proposition \ref{P: peak point if red saturated}(1), there is $f$ such that $P(\abs{f})=\{z\}$, hence the conclusion. Otherwise it is not isolated, so $z$ is in the closure of its complement. As $\{z\}^{\complement}=\bigcup_{z\in P(\abs{f})}P(\abs{f})^{\complement}$, thanks to the Urysohn-Fréchet property, one can choose a sequence of points $\{z_j\}_{j\in \bbN}$ in the complement converging to $z$. Finally, choose inductively $\{f_j\}_{j\in \bbN}$ (for $\bbN\subset I$) accordingly: $f_{j+1}$ is chosen such that $z_j\notin P(\abs{f_{j+1}})$ and $P(\abs{f_{j+1}})\subseteq P(\abs{f_{j}})$. Thus $I$ can be replaced by $\bbN$.
\end{proof}

The important topological property of Berkovich analytic space used above is
\begin{theorem}\label{T: angelique}\cite{Fav}\cite{Poi}
    Berkovich spaces (over n-A fields) are Urysohn-Fréchet spaces, i.e., for any subset, any point in its topological closure can be realized as the limit of a sequence of points in this subset.
\end{theorem}
\begin{corollary}
    If $\KB$ contains a dense finitely generated $k$-algebra $B$, then $\KM(\KB)$ is an Urysohn-Fréchet space.
\end{corollary}
\begin{proof}
    The natural map $\KM(\KB)\to (\spec B)^{\an}$ realizes the former as a topological subspace of the later, and being an U-F space is inherited by subspaces. 
\end{proof}

\section{Special case: Fubini-Study metric}
As mentioned in the introduction, if the metric $\phi$ is a Fubini-Study metric, then the main result have been obtained with far less complicated methods. We record them here to illustrate better the problem. For simplicity, we shall assume that $L$ is very ample and $s\in R_1$, and that $X$ is irreducible of dimension $d$.



\subsection{Calculation by affinoid algebras}
If $\phi$ is Fubini-Study, then under our assumptions on the base field (see \S1.4), the Banach section algebra $\KR_{\phi}$ is an affinoid algebra. The analytic localization procedure can be performed simply in a finite morphism of affinoid algebras, where all Shilov points are organized in a fixed global Noether normalization. In this case the determinant distorsion can be expressed as the norm in the field arithmetic for such normalization. 
\begin{construction}\label{Cs: Noether normalization section algebra}
    \cite[\S6]{BGR} As $\KR_{\phi}$ is a graded affinoid algebra, by Noether normalization, one can find a finite morphism of affinoid algebras from the Tate algebra $\nu: \KT_{d+1}\to \KR_{\phi}$ which preserves the grading. We shall omit the subscripts and denote these two Banach algebras by $\KT$ and $\KR$. 
    
    The field extension $\Frac(\KT)\to \Frac(\KR)$ is finite, of degre $e$. On $\Frac(\KR)$ there are finitely many inequivalent absolute values $\{\abs{\ndot}_{a}\}$ that are extensions of the Gauss absolute value $\abs{\ndot}_{\gamma}$ on $\KT$. The inequivalent absolute values are in fact indexed by the irreducible components of $\widetilde{\KR}$, hence by $\eta(\widetilde{\KR})$, which is the set of Shilov points. Then there is a decomposition $\norm{f}_{\Sp,\KR}=\max_{a\in \Sh}\abs{f}_a$. 
\end{construction}

\begin{construction}\label{Cs: determinant as field norm}
    \cite[\S3]{BGR} For any $a$, denote by $\KR_a$ and $\Frac(\KR)_a$ the completions with respect to $\abs{\ndot}_a$, then the induced extension of complete valued fields $\nu_a:\Frac(\KT)\to \Frac(\KR)_a$ is still finite, of degree $e_a$, and one has $\sum e_a=e$. Denote by $\nu$ and by $\nu_a$ the (algebra) norm maps $\Frac(\KR)\to \Frac(\KT)$ and $\Frac(\KR)_a\to \Frac(\KT)$. The $\Frac(\KT)$-module $\oplus_a \Frac(\KR)_a$ is finite. There are equalities $\nu(f)=\prod_a\nu_a(f)$ and $\abs{\nu_a(f)}_{\gamma}=\abs{f}_a^{e_a}$, thus a decomposition $\abs{\nu(f)}_{\gamma}=\prod_a\abs{f}_a^{e_a}$. 

    Denote by $\KR'_a$ the normalization of the image of $\KR$ in $\Frac(\KR)_a$. Then $\oplus_a\KR'_a$ is a finite $\KT$-module; moreover it is still graded. The $k$-codimension of $\KR_{\leq n}$ in $(\oplus_a\KR'_a)_{\leq n}$ is of order $O(n^d)$, with $\lim\frac{1}{\upsilon(n)}\dim_k (\oplus_a\KR'_a)_{\leq n}=e_a$ where $\upsilon(n)=\sum_{i\leq n} \chi(i)$ is the $k$-linear dimension of $\KR_{\leq n}$.
\end{construction}

\begin{remark}
    As in the general case, one can calculate the limit of the `determinant norm' of the $k$-linear map $(s\cdot)$ on $\KR$, namely $\frac{1}{\chi(n)}\log \norm{\bigwedge^{\chi(n)}(\cdot s)}_{\hom,\phi}$, on $\oplus_a\KR'_a$ instead of on $\KR$ thanks to Proposition \ref{P: det norm distorsion subspace} and using Proposition \ref{P: VF extension operator norm}. The result being $\sum_a\frac{e_a}{e}\log\abs{s}_a$, is in fact equal to the norm $\log~\abs{\nu(s)}_{\gamma}^{1/e}$. Note that the algebraic norm $\nu(s)$ is just the determinant of the operator $(s\cdot)$, but viewed as a $\Frac(\KT)$-linear map on the vector space $\Frac(\KR)$. Thus this is a normed/metrized/arithmetic version of the formula expressing the algebraic Hilbert-Samuel multiplicity by means of Noether normalization.
    
    This result follows from the viewpoint that for model metrics $\phi$, the measure $\MA(\phi)$ (defined using differential calculus) has a nice local algebraic incarnation, see \cite[\S6]{CLD} and \cite[\S10]{GK}; also, it is compatible with the viewpoint in \cite[\S8]{BE} and \cite{CM}, treating Deligne's metrized intersection pairing $\langle (L,\phi),\dots,(L,\phi)\rangle$ as a norm functor on the affinoid Banach section algebra of $(L,\phi)$ and relating to the Monge-Ampère measure $\MA(\phi)$.

    This is a consequence of the `regularity' of the metric (i.e. being Fubini-Study, or coming from an integral model); for general Shilov finite metrics, we lose the possibility of affinoid Noether normalization, even locally around each Shilov point, hence the argument of \cite[\S6]{CLD} is not directly applicable. Thus our effort in the main body of the article is to relax such a `global regularity' assumption.
\end{remark}

\bibliographystyle{abbrv}

\end{document}